\documentclass[hidelinks,onefignum,onetabnum]{siamart220329}

\usepackage{lipsum}
\usepackage{amsfonts}
\usepackage{graphicx}
\usepackage{epstopdf}
\usepackage{algorithmic}
\usepackage{enumitem}
\usepackage{subcaption}
\ifpdf
  \DeclareGraphicsExtensions{.eps,.pdf,.png,.jpg}
\else
  \DeclareGraphicsExtensions{.eps}
\fi

\newsiamremark{remark}{Remark}
\newsiamremark{Prop}{Proposition}
\newsiamremark{hypothesis}{Hypothesis}
\crefname{hypothesis}{Hypothesis}{Hypotheses}
\newsiamthm{claim}{Claim}

\headers{Consensus-based optimization via optimal feedback control}{}

\title{Fast and robust consensus-based optimization via optimal feedback control%\thanks{Submitted to the editors DATE.
}

\author{Yuyang Huang\thanks{Department of Mathematics, Imperial College London, South Kensington Campus SW72AZ London, UK \email{\{yuyang.huang21,dkaliseb,n.kantas\}@imperial.ac.uk}.}\and Michael Herty\thanks{IGPM, RWTH Aachen University, Templergraben 55, D-52062 Aachen, Germany; \email{herty@igpm.rwth-aachen.de}.}\and Dante Kalise\footnotemark[1]\and Nikolas Kantas\footnotemark[1]}

\usepackage{amsopn}

\usepackage{todonotes}
\usepackage{xargs} 
\newcommandx{\change}[2][1=]{\todo[linecolor=blue,backgroundcolor=blue!25,bordercolor=blue,#1]{#2}}
\usepackage{amssymb}

\ifpdf
\hypersetup{
  pdftitle={Fast and robust consensus-based optimization via optimal feedback control},
  pdfauthor={Yuyang Huang, Michael Herty, Dante Kalise, Nikolas Kantas}
}
\fi

\begin{document}

\maketitle

\begin{abstract}
We propose a variant of consensus-based optimization (CBO) algorithms, \textsl{controlled-CBO}, which introduces a feedback control term to improve convergence towards global minimizers of non-convex functions in multiple dimensions. The feedback law is a gradient of a numerical approximation to the Hamilton-Jacobi-Bellman (HJB) equation, which serves as a proxy of the original objective function. Thus, the associated control signal furnishes gradient-like information to facilitate the identification of the global minimum without requiring derivative computation from the objective function itself. The proposed method exhibits significantly improved performance over standard CBO methods in numerical experiments, particularly in scenarios involving a limited number of particles, or where the initial particle ensemble is not well positioned with respect to the global minimum. At the same time, the modification keeps the algorithm amenable to theoretical analysis in the mean-field sense. The superior convergence rates are assessed experimentally.
\end{abstract}

% REQUIRED
\begin{keywords}
global optimization, consensus-based optimization, Hamilton–Jacobi–Bellman PDEs, high-dimensional polynomial approximation
\end{keywords}

% REQUIRED
\begin{MSCcodes}
65K10, 35Q93, 90C56
\end{MSCcodes}

\section{Introduction}
Optimization plays an essential role in modern science, engineering, machine learning, and statistics \cite{goffe1994global,bottou2018optimization,kulkarni2015particle}. Over the last decades, the flourishing of computational data science has brought significant challenges to optimization, such as high-dimensional search spaces, and  non-convex, non-smooth objective functions. These features often limit the use of traditional gradient-based optimization methods \cite{lan2020first,nesterov2018smooth}. For general optimization problems of the form
\begin{equation}
\label{min}
\min _{x \in \mathbb{R}^d} f(x),
\end{equation}
where $f:\mathbb{R}^d\rightarrow\mathbb{R}$ is continuous, bounded from below,  and attains a unique global minimum,
the objective function is possibly  nonconvex and has many local minimizers. In such settings, metaheuristic methods have been proposed as an effective alternative for addressing the challenges of global optimization. Among these metaheuristics, agent-based algorithms such as Ant Colony Optimization, Particle Swarm Optimization, or Simulated Annealing \cite{wang2018particle,dorigo2006ant,bertsimas1993simulated,Ding2024SwarmbasedGD} have demonstrated remarkable performance when applied to NP-hard problems. More recently, a novel class of multi-particle, derivative-free methods, known as  \emph{Consensus-based Optimization} (CBO), has been introduced in \cite{pinnau2017consensus}. Through a combination of probabilistic and mean field arguments, these methods are capable of effectively solving non-smooth and non-convex global optimization problems in multiple dimensions.

The CBO method with anisotropic diffusion employs a system of $N \in \mathbb{N}$ interacting particles with position vector $X_t^i \in \mathbb{R}^d, i=1, \ldots, N$, evolving in time $t\in[0,\infty)$ according to a system of  stochastic differential equations (SDEs):
\begin{equation}
\label{sde1}
d X_t^i=-\lambda\left(X_t^i-v_\alpha(\rho_t^N)\right) d t+\sigma \operatorname{Diag}(X_t^i-v_\alpha(\rho_t^N)) d W_t^i,
\end{equation}
where $\lambda,\sigma>0$  are drift and noise parameters, respectively. The operator $\operatorname{Diag}:\mathbb{R}^d\rightarrow\mathbb{R}^{d\times d}$ maps a vector $\nu\in\mathbb{R}^d$ onto a diagonal matrix with elements of $\nu$, and $\left((W_t^i)_{t\geq 0}\right)_{i=1,\cdots,N}$ are i.i.d Wiener processes in $\mathbb{R}^d$.
We assume that the initial condition of the particles $X_0^i \in \mathbb{R}^d$  are i.i.d with $\operatorname{law}\left(X_0^i\right)=\rho_0 \in \mathcal{P}\left(\mathbb{R}^d\right)$, where the set $\mathcal{P}\left(\mathbb{R}^d\right)$ contains all Borel probability measures over $\mathbb{R}^d$. The particle dynamics are driven by two forces: a drift term, forcing particles to move towards the consensus point $v_\alpha(\rho_t^N)$, and a diffusion term, allowing random exploration of the search space, see \cite{pinnau2017consensus,totzeck2022trends} for a detailed description. The consensus point $v_\alpha(\rho_t^N)$ is calculated by the weighted average
\begin{equation}
\label{v_f}
v_\alpha(\rho_t^N):=\frac{1}{\sum_{i=1}^N \omega_f^\alpha\left(X_t^i\right)} \sum_{i=1}^N  X_t^i \omega_f^\alpha\left(X_t^i\right),%=\frac{\int x\;\omega_f^\alpha(x) d \rho_t^N}{\|\omega_f^\alpha\|_{L_1(\rho_t^N)}}
\end{equation}
where we denote by $\rho_t^N$
the empirical measure $\frac{1}{N}\sum_{i=1}^N\delta_{X_t^i}$. The weight  $\omega_f^\alpha$ is defined as
\begin{equation}
\label{weight}
\omega_f^\alpha(x)=\exp (-\alpha f(x)), \quad \alpha>0,
\end{equation}  
which assigns higher weights to lower objective function values $f(\cdot)$. As a result, particles with lower function values have a stronger influence on both $v_\alpha(\rho_t^N)$ and the motion of the rest of the particles. Due to Laplace's principle  from large deviations theory \cite{dembo2009large}, we expect $v_\alpha(\rho_t^N)$ to approximate the global minimum of the particle system when $\alpha$ is large enough, i.e.,
$$
\lim_{\alpha\rightarrow\infty}v_\alpha(\rho_t^N) \approx \operatorname{argmin}_{i=1, \ldots, N} f\left(X_t^i\right).
$$
A very relevant feature of CBO methods, as opposed to traditional metaheuristics, is the fact that they are amenable to rigorous convergence analysis in the mean-field, long-term evolution, limit \cite{carrillo2018analytical,huang2022mean,fornasier2024consensus,riedl2023gradient}. Recent extensions of CBO include constrained optimization \cite{borghi2023constrained,fornasier2021consensus,beddrich2024constrained,carrillo2023consensus}, multi-objective optimization \cite{borghi2022consensus,klamroth2024consensus} and multi-level optimization~\cite{herty2024multiscale,trillos2024cb,borghi2024particle}. Also,  variants such as CBO 
 with jump-diffusions \cite{kalise2022consensus}, adaptive momentum \cite{chen2020consensus}, personal best information \cite{totzeck2020consensusbased} and
polarization  \cite{bungert2022polarized} have been proposed to improve  its performance. In this paper, we will study a novel variant of the CBO algorithm, which we refer to as \textsl{controlled-CBO}. By borrowing a leaf from optimal control theory, we enhance robustness and convergence towards global minimizers by introducing a feedback control term.
 
\subsection*{A control-theoretical approach to global optimization}
The connection between control theory and global optimization has gained considerable attention over the last years. Most notably, in Chaudhari et al. \cite{chaudhari2018deep}, Hamilton-Jacobi-Bellman (HJB) PDEs arising from dynamic programming and optimal feedback control have been studied as a tool for convexification. More recently, Bardi and Kouhkouh \cite{bardi2023eikonal}
formulated an eikonal-type HJB equation for global optimization using weak Kolmogorov-Arnold-Moser (KAM) theory and a small-discount approximation. The solution, represented as the value function of an optimal control problem, yields optimal trajectories that converge to a global minimizer of the problem \eqref{min} without computing the gradient of the objective function.

In the following, we present a deterministic,  infinite horizon, discounted formulation of the global-optimization-as-optimal-control approach. Consider a control system where the control variable $u(\cdot)$ governs the state trajectory  $y(t)\in \mathbb{R}^d$
through the dynamics
\begin{equation}
\label{control}
\dot{y}=u(t), \quad u(\cdot) \in \mathcal{U}, \quad y(0)=x.
\end{equation}
Here, the control lies in the set $\mathcal{U}:=\left\{u(t): \allowbreak\mathbb{R}_+ \mapsto \mathbb{R}^d, \text {Lebesgue measurable a.e. in } t\right\}$ and $x\in\mathbb{R}^d$ is a given initial condition. To quantify the performance of the control, we consider an objective function
$$
\mathcal{J}\left(u(\cdot), x\right):=\int_0^{\infty}e^{-\mu t} 
\left(f\left(y(t)\right)+\frac{\epsilon}{2}|u(t)|^2\right) d t,
$$
where $\mu>0$ is a discount factor and $\epsilon>0$ is a
parameter for Tikhonov regularization. Then, the optimal control is obtained by solving the time-homogeneous, discounted infinite horizon optimal control problem
\begin{equation}
\label{eq:control_prob}
\min _{u(\cdot) \in \mathcal{U}} \mathcal{J}\left(u(\cdot), x\right)\quad\text{subject to }\cref{control}.
\end{equation}
It is well-known (e.g. \cite[Section II.11]{fleming2006controlled}) that the optimal value function
\begin{equation*} 
V\left(x\right)=\inf _{u(\cdot) \in \mathcal{U}} \mathcal{J}\left(u(\cdot), x\right)
\end{equation*}
 is the unique viscosity solution to the stationary Hamilton-Jacobi-Bellman (HJB) equation:
\begin{equation}
\label{HJB}
-\mu V(x)+\min_{u \in \mathbb{R}^d}\left\{D V(x)^\top u+f(x)+\frac{\epsilon}{2}|u|^2\right\}=0,
\end{equation}
with $D V=\left(\partial_{x_1} V, \ldots, \partial_{x_d} V\right)^\top $.
Once the HJB equation \cref{HJB} is solved, the optimal control $u^*$ of \cref{HJB} is given in feedback form by
\begin{equation}
\label{optimal control}
u^*(x):=\underset{u \in \mathbb{R}^d}{\operatorname{argmin}}\left\{D V(x)^\top  u+f(x)+\frac{\epsilon}{2}|u|^2\right\}=-\frac{1}{\epsilon}  D V(x),
\end{equation}
which depends only on the current state 
$x=x(t)$ (the initial condition for the remaining horizon) along a trajectory.
 The control law utilizes gradient information $DV$ from the value function, obtained from the solution of the HJB PDE, rather than directly fetching the gradient of the objective $f$. 
 \begin{remark}
There are interesting connections between the optimal control problem defined in \cref{control}-(\ref{optimal control}) and the global minimization \eqref{min} at the limit when the discount factor $\mu \rightarrow 0$. In this regime, \cite{bardi2023eikonal} shows (for $\varepsilon=$ 1) that the PDE (\ref{HJB}) converges to $\min f+\frac{1}{2}|DV(x)|^2=f(x)$, and the system governed by $\dot{y}=-\frac{D V}{|D V|}$ converges to a global minimizer in finite time. Although establishing a rigorous connection with our approach lies beyond the scope of this work, one may interpret the control term in \cref{optimal control}, for small values of $\mu>0$,  as a relaxation of the strategy proposed in \cite{bardi2023eikonal}. From this perspective, the use of $D V$ in the feedback control offers an advantage over classical gradient descent based on $D f$, as it is more capable of transcending local minima. As such, the control $u^*$ derived from $D V$ behaves as a more promising candidate for guiding the CBO dynamics.  On the other hand when $\mu=\epsilon=0$, and the control action is restricted to the unit ball, the HJB equation \cref{HJB} simplifies to $$\underset{\|u\|\leq 1}{\min}\left\{D V(x)^\top u+f(x)\right\}=0,$$leading to an Eikonal type equation $|DV|=f$ and to a normalized gradient decent $u^*=-\frac{DV}{|DV|}$.
\label{rem:1}
\end{remark}

While this framework provides valuable theoretical insights, its practical implementation remains unexplored. Most notably, it requires the numerical approximation of a $d-$dimensional HJB PDE. The accurate numerical implementation of direct (“data‐free”) solvers for high‐dimensional, stationary HJB PDEs remains a formidable challenge. Sparse‐grid \cite{bokanowski2013adaptive,kang2017mitigating} techniques yield rigorous convergence under suitable regularity assumption and are practically feasible for systems up to dimension 8. Low‐rank tensor decomposition methods \cite{dkk21} have been shown to solve HJB PDEs in dimensions exceeding 100, but the convergence analysis
for nonlinear dynamics remains unresolved. 
In the present work, we use 
polynomial approximation \cite{Kalise_2018} solved via Galerkin projection. This choice gives closed‐form expressions for the surrogate value function and enables a fully rigorous convergence analysis of the approximation.  When combined with hyperbolic cross basis \cite{beck2012optimal,chkifa2015breaking}, the method can be applied to problems  with dimensionality up to 80 at moderate computational cost \cite{azmi2021optimal}.

However, such methods naturally introduce numerical discretization errors that will affect the convergence of the optimal trajectories towards the global minimizer. In this paper, we bridge this gap between numerical discretization errors and global optimality by augmenting the standard CBO method with the resulting state feedback law $u^*(x)$, derived from the optimal control problem \cref{eq:control_prob}. Revisiting the system \eqref{sde1}, we introduce the following controlled-CBO dynamics:
  \begin{equation}
  \label{eq:cbo_intro}
d X_t^i=[-\lambda\left(X_t^i-v_\alpha(\rho_t^N)\right)+\beta u^*(X_t^i)]d t+\sigma \operatorname{Diag}\left(X_t^i-v_\alpha(\rho_t^N)\right) d W_t^i\,
\end{equation}
with a parameter $\beta>0$ controlling the strength of the control. In addition to the standard drift in \cref{sde1},  the feedback control provides gradient-like information, effectively guiding the particles toward the minimizer of the objective function.  Note that the controlled-CBO method remains gradient-free, in the sense that no gradient of the objective function is required. This is particularly relevant for applications related to shape optimization, where the computations of gradients is computationally demanding \cite{KKS}. Our approach is similar in spirit to that of  Schillings et al. \cite{schillings2023ensemble}, where ensemble techniques and generalized simplex gradients (referred to ensemble-based gradient inference) are used to introduce an additive guiding term into the CBO dynamics.  
\subsection*{Main contributions of this work}
We highlight the most relevant contributions of this paper:
\begin{enumerate}
    \item The controlled-CBO method significantly enhances the performance of standard CBO algorithms. It achieves faster convergence rates and improved robustness, especially in scenarios where the number of particles is limited and there is a lack of good prior knowledge for initializing the particle system.     
    \item The proposed methodology demonstrates the potential of control-based meth\-od as a significant advancement in solving complex, high-dimensional, non-convex optimization problems. While previous work \cite{bardi2023eikonal} established a theoretical foundation for using HJB equations in global optimization, it did not offer a numerically feasible solution due to the complexity of solving HJB equations. Our method overcomes this challenge by applying the successive approximation algorithm \cite{Kalise_2018} to solve the HJB equations  in infinite-horizon cases and utilising an interacting particle system to correct the approximation errors.
    \item Additionally, we conduct a thorough convergence analysis of the successive approximation algorithm and provide rigorous proofs of the well-posedness of the resulting interacting particle system. 
\end{enumerate}

The solution of the HJB PDE  \eqref{HJB} can be viewed as a proxy for the objective function $f$, and it has value in its own right from both practical and theoretical perspectives. Here, we emphasize that combining it with the CBO method gives a numerically robust approach in terms of overall accuracy and scalability for a given computational budget. Even a coarse, relatively inexpensive approximation of $V$ is sufficient to steer the dynamics quickly toward a favorable region, and then CBO can refine the search and achieve high accuracy.  
The rest of the paper is structured as follows. In \Cref{num_HJB}, we study the numerical approximation of the HJB PDE \eqref{HJB} using a successive approximation algorithm and high-dimensional polynomial basis. After obtaining the approximate feedback control law, in \Cref{control_CBO}, we present the controlled-CBO algorithm, discussing the well-posedness of the controlled-CBO dynamics. 
In \Cref{numerical2} we provide extensive numerical experiments on classic benchmark optimization problems.

\section{Numerical approximation of the HJB PDE}
\label{num_HJB}
In this section, we introduce a method to approximate the value function $V$ and the optimal feedback map $u^*$ from the HJB equation \eqref{HJB}. We apply a successive approximation algorithm in the same spirit as in \cite{beard1997galerkin,beard1998approximate,Kalise_2018}. Note that this algorithm corresponds to a continuous-in-space version of the well-known policy iteration algorithm~\cite{kundu2024policy}. These algorithms can be interpreted as a Newton iteration to address the nonlinearity present in \eqref{HJB}.  As such, a fundamental building block is the solution, at the $m$-th ($m\in\mathbb{N})$ iteration of the method
given a fixed control law $u^{(m)}(x)$, of the linear Generalized HJB equation (GHJB) for $V^{(m)}$:
\begin{equation}
\begin{aligned}
\label{GHJB}
&\mathcal{G}_\mu(V^{(m)}, D V^{(m)} ; u^{(m)})=0,\\
&\mathcal{G}_\mu(V, D V ; u):=-\mu V+D V^\top u+f+\frac{\epsilon}{2}|u|^2.
\end{aligned}
\end{equation}
Having computed the value function $V^{(m)}$, an improved feedback law is obtained as $u^{(m+1)}=-\frac{1}{\epsilon}DV^{(m)}$ from equation \cref{optimal control}, and we iterate via \eqref{GHJB}. Throughout the iterative processes, the solution of GHJB equation converges uniformly to the solution of HJB equation \cref{HJB}; see \cite{beard1997galerkin,beard1995improving} for details.
In order to guarantee the existence of a solution to the GHJB equation, we require that at every iteration, $u^{(m)}$ is an admissible feedback control, in the sense of \Cref{def:admissible}:\begin{definition}
\label{def:admissible}
A feedback control $u:=u(x)$ is admissible on $\Omega\subset\mathbb{R}^d$, written $u\in\mathcal{A}(\Omega)$, if $u$ is continuous on $\Omega$ and $\mathcal{J}(u(\cdot),x)<\infty$ for any $x\in\Omega$.
\end{definition}
We refer to \Cref{lemma_existence} in \Cref{sec:appx_vi} for the role that the admissibility condition plays in the existence of a solution to the GHJB equation. Despite being a linear equation, a general closed-form solution of the GHJB equation remains elusive. Therefore, we approximate the GHJB equation using a Galerkin method with global polynomial basis functions, and provide sufficient conditions for the convergence of the method.

The numerical approximation of the GHJB equation begins with the selection of set of (not necessarily linearly independent) continuously differentiable basis functions  $\Phi_n(x)=\left\{\phi_i(x)\right\}_{i=1}^{n}$ of $\mathcal{L}^2(\Omega)$, where each $\phi_i\in\mathcal{L}^2(\Omega\,;\mathbb{R})$. We approximate the solution $V_n$ to the \Cref{GHJB} by a
Galerkin projection:
$$
V_n(x)=\sum_{i=1}^n c_i \phi_i(x)=\Phi_n(x)^\top\mathbf{c}_n\,,
$$
and determine the coefficients $\mathbf{c}_n=\{c_i\}_{i=1}^{n}$  by solving a system of residual equations for a given admissible control $u$
\begin{equation}
\langle \mathcal{G}_\mu\left(V_n, DV_n; u\right),\phi_i\rangle:=
%\int_{\Omega}\mathcal{G}_\mu\left(V_n, DV_n; u\right) \phi_i(x)\; d x=
\int_{\Omega}\mathcal{G}_\mu\left(\Phi_n^\top\mathbf{c}_n, \nabla\Phi_n^\top\mathbf{c}_n; u\right) \phi_i(x)\; d x=0\,,\quad 1\leq i\leq n.
\label{Eq:residual}
\end{equation}
For the sake of simplicity, and with a slight abuse of notation, we define the vector-valued inner product against $\Phi_n$ as $\left\langle \mathcal{G}_\mu, \Phi_n\right\rangle=\left(\left\langle \mathcal{G}_\mu, \phi_1\right\rangle, \ldots,\left\langle \mathcal{G}_\mu, \phi_n\right\rangle\right)^{\top},$ where $\mathcal{G}_\mu$ is shorthand for $\mathcal{G}_\mu\left(V_n, DV_n; u\right)$ and each $\left\langle\mathcal{G}_\mu, \phi_i\right\rangle$  represents the scalar inner product as defined in \cref{Eq:residual}. The approximated feedback control $u_n$  is then recovered as:
$$
u_n(x)  =-\frac{1}{\epsilon}D V_n(x) \\ =-\frac{1}{\epsilon} \nabla \Phi_n(x)^\top \mathbf{c}_n.
$$
The resulting successive approximation algorithm is presented in \cref{alg1}. Starting from an admissible initial control $u^{(0)}$, at each iteration, the control law is updated based on the gradient of the approximated value function. 
After attaining a preset tolerance, the algorithm returns an approximation $u_n^{(m+1)}$ to $u^*$. The output
$u_n^{(m+1)}$ will be then used as  forcing term for the controlled-CBO in \eqref{eq:cbo_intro}. 
\begin{algorithm}
\small
\begin{algorithmic}[1]
\STATE $\textbf{Given: } \mu>0,\; \epsilon \in(0,1),\; tol>0,$ fix $\Omega,\;\Phi_n$.
\STATE  \textbf{Input} $u_n^{(0)}\equiv u^{(0)}\in \mathcal{A}(\Omega)$  \hfill \% initialization with  $m=0$
\STATE \textbf{While } $\|u_n^{(m+1)}-u_n^{(m)}\|>tol$\textbf{ do }
\STATE    \qquad \textbf{Solve} 
$\left\langle\mathcal{G}_\mu\left(\Phi_n^\top\mathbf{c}_n^{(m)}, \nabla\Phi_n^\top\mathbf{c}_n^{(m)}; u_n^{(m)}\right),\Phi_n\right\rangle=0$\\
 \hfill \% compute $\mathbf{c}_n^{(m)}$ by solving \Cref{Eq:residual} 
\STATE    \qquad \textbf{Obtain} $V_n^{(m)}(x)=\mathbf{c}_n^{(m)}\Phi_n(x)$ \hfill \% Obtain value function\\ 
\STATE    \qquad \textbf{Update} $u_n^{(m+1)}(x)=-\frac{1}{\epsilon}D V_n^{(m)}(x)$ \hfill \% update control function
\STATE \textbf{End While}
\STATE \textbf{Return} $u^*=u_n^{(m+1)}(x)$ and $V^*=V_n^{(m)}(x)$ \hfill \% output
\end{algorithmic}
\caption{}
\label{alg1}
\end{algorithm}

\paragraph{Convergence} 
We adapt sufficient conditions from \cite{beard1997galerkin,beard1998approximate,beard1995improving} to establish the convergence of the iterative scheme and of $V_n^{(m)}$.
\begin{assumption}
    \begin{enumerate}[label=(\roman*)]    
        \item The domain $\Omega$ is a compact set over which $f$ is non-negative and Lipschitz continuous.
        \label{1}
        \item The initial control $u^{(0)}\in\mathcal{A}(\Omega)$.
        \label{2}
        \item 
       Given arbitrary control $u\in\mathcal{A}(\Omega)$, the solution of $\mathcal{G}_\mu\left(V,DV; u\right)=0$ satisfies $V \in \operatorname{span}\left\{\phi_i\right\}_{i=1}^{\infty} \subseteq \mathcal{L}^2(\Omega)$.
        \label{3}
        \item For any $1\leq i,j\leq n$, $u\in\mathcal{A}(\Omega)$ we have $\frac{\partial \phi_i^\top}{\partial x} u,\;\left|u\right|^2,\; \frac{\partial \phi_i^\top}{\partial x} \frac{\partial \phi_j}{\partial x}$ are continuous and in $\operatorname{span}\left\{\phi_i\right\}_{i=1}^{\infty}$. %$u$ is an arbitrary admissible control.
        \label{4}
        \item $\sum_{i=1}^\infty c_i\phi_i$ and $\sum_{i=1}^\infty c_i\frac{\partial \phi_i}{\partial x}$ converge uniformly to $V$ and $\frac{\partial V}{\partial x}$ on $\Omega$, respectively.
\label{5}
  \item The coefficients $\{c_i\}_{i=1}^n$ of $V_n=\sum_{i=1}^nc_i \phi_i$ are uniformly bounded for all $n$. 
\label{c_bdd}
        \item For any $u\in\mathcal{A}(\Omega)$, there exists $x\in\Omega$ such that $\left(-\mu \phi_i +\frac{\partial \phi_i}{\partial x} \cdot u\right)(x)\neq 0$, $1\leq i\leq n$.
    \label{new}
\item $\sum_{i=1}^{\infty}\left\langle f+\frac{\epsilon}{2}\left|u\right|^2, \phi_i\right\rangle \phi_i
$, $\sum_{i=1}^{\infty}\left\langle \phi_k, \phi_i\right\rangle \phi_i
$  and $\sum_{i=1}^{\infty}\left\langle\frac{\partial \phi_k}{\partial x} u, \phi_i\right\rangle \phi_i$ are\\ \textit{pointwise decreasing} \footnote{A pointwise convergent infinite sequence $\sum_{i=1}^{\infty} c_i \phi_i(x)$ on $\Omega$ is called pointwise decreasing, if $\forall \;k\in\mathbb{N}$, $\forall \varepsilon>0$, there exists $\rho>0$ and $m>0$ such that $\forall x \in \Omega$, then $n>m$ and $\left|\sum_{i=k+1}^{\infty} c_i \phi_i(x)\right|<\rho$ imply that
$
\left|\sum_{i=k+n+1}^{\infty} c_i \phi_i(x)\right|<\varepsilon 
$ (see \cite[Definition~17]{beard1997galerkin}).} for any 
 $u\in\mathcal{A}(\Omega)$, $k=1,2,\cdots$.
\label{6}
    \end{enumerate}
    \label{ass_convergence}
\end{assumption}
We note that although the dynamics \eqref{control} are defined on $\mathbb{R}^d$, the compact domain $\Omega\subset\mathbb{R}^d$ in \Cref{assumption}\ref{1} is introduced for the numerical approximation of the value function $V$. In practice, $\Omega$ is chosen to be sufficiently large and since the dynamics are linear in the control $u(\cdot)$, it is always possible to construct admissible controls that ensure the trajectories to remain within $\Omega.$ 
The convergence of $V_n^{(m)}$ to the solution of the HJB equation has been thoroughly studied   for the \emph{un-discounted} infinite horizon case in \cite{beard1995improving,beard1997galerkin,kundu2024policy, beard1998approximate}, in which the admissible control is further required to  asymptotically stabilize the system around zero. However, in our case, the inclusion of the discount factor $\mu$ simplifies such admissibility conditions. We present a convergence result similar to \cite[Theorem 4.2]{beard1998approximate}.
\begin{proposition}
\label{main}
    Denote the solution to the HJB equation (\ref{HJB}) by $V^*$. Assume Assumption \ref{ass_convergence} holds; then for any $\delta>0$, there exists $M,N\in\mathbb{N}$ such that $n>N$ and $m>M$ imply $\vert V^{(m)}_n - V^*\vert_{L^2(\Omega)}<\delta$ with $u_n^{(m+1)}\in\mathcal{A}(\Omega)$.
\end{proposition}
\begin{proof}
    See \cref{sec:appx_vi}.
\end{proof}
 
\paragraph{Tuning the initialization}
Ensuring closed-loop stability via  stabilizing feedback controls is a fundamental and well-studied topic in control theory. The convergence of \Cref{alg1} requires selecting an initial control $u^{(0)}:\Omega\rightarrow\mathbb{R}^d$, such that the resulting closed-loop system $\dot{y}=$ $u^{(0)}(y(t;x,u^{(0)}))$, $y(0)=x$,
is asymptotically stable in the sense of Lyapunov (see \cite[Chapter 4]{evans1983introduction}) for every initial conditions $x\in\Omega$.  Here  $y(t;x,u^{(0)})$ denotes the solution to the system \eqref{control} with $u^{(0)}$. However, selecting an appropriate stabilizing control $u^{(0)}$ can be a challenge in practice. A common approach is to exploit the 
nature of the infinite horizon problem and tune instead the discount factor. This is presented in \cref{alg2} as an outer loop for \cref{alg1} inspired from optimal feedback stabilization problems.  For a sufficiently large $\mu$, the successive approximation algorithm can be initialized in \cref{alg2} with $u^{(0)}=0$ (see \cite{Kalise_2018}).
\begin{remark}
The initial control $u^{(0)}(\cdot) \equiv 0$ is admissible. Since $\dot{y} \equiv 0$ and $y(0)=x$, it is clear that  $\mathcal{J}\left(u^{(0)}(\cdot), x\right)=\int_0^{\infty} e^{-\mu t} f\left(x\right) d t<\infty$. It is also worth noting that as $u^{(0)}(\cdot) \equiv 0$, the first iteration yields the GHJB equation
$$
-\mu V+f=0,
$$
which implies that the updated control $u^{(1)}=-\frac{1}{\mu\epsilon} \nabla f$, therefore, at the beginning of the successive approximation algorithm, the feedback law yields a gradient descent update for  the objective $f$.
\end{remark}
\begin{algorithm}
\small
\begin{algorithmic}[1]
\STATE $\textbf{Given:} \mu>0, tol_\mu>0, \theta\in(0,1),$ fix $\Omega,\;\Phi_n$.
\STATE  \textbf{Input} $u^{(0)}=0$  \hfill \% initialization
\STATE \textbf{While } $\mu> tol_\mu$ \textbf{ do } \hfill \% Outer loop
\STATE \qquad\textbf{Obtain} $(V_n^{(m)},u_n^{(m+1)})$ by\cref{alg1} initialized with $u^{(0)}$.\hfill \% Inner loop
\STATE \qquad\textbf{Update} $u^{(0)}=u_n^{(m+1)}$, $\mu=\theta\mu$
\STATE \textbf{End While} 
\STATE \textbf{Return} $(V^*,u^*)=(V_n^{(m)},u_n^{(m+1)})$ 
\hfill \% output
\end{algorithmic}
\caption{}
\label{alg2}
\end{algorithm}
Although solving the HJB equation, particularly for high-dimensional dynamics, may overwhelmingly expensive, it is important to note that 
this is an offline phase to be performed once. Once solved, the resulting optimal feedback control can be introduced into the controlled-CBO framework. In the following, we focus on the assembly and solution of the Galerkin residual equations \eqref{Eq:residual} in high dimensions.

\subsection*{Building a global polynomial basis for the value function}
For $i=1,\ldots,$$n$, $j=1,\ldots,d$, let $\phi_i^j:\mathbb{R}\rightarrow\mathbb{R}$ denote a one-dimensional polynomial basis of $\mathcal{L}^2(\Omega_j)$, where $\Omega=\bigotimes_{j=1}^d\Omega_j$. For the sake of simplicity, we consider a  monomial basis, but the idea extends to any orthogonal basis, e.g.  Legendre polynomials, where each basis element of $\Phi_n:=\left(\phi_1(x), \ldots, \phi_n(x)\right)$ admits a separable representation, i.e.,
$$
\phi_i(x)=\prod_{j=1}^d \phi_i^j\left(x_j\right)=\prod_{j=1}^d x_j^{r_i^j} \text{ with } \textbf{r}_i=(r_i^1,r_i^2,\cdots,r_i^d)\in\mathbb{N}_0^d.
$$
We generate $\Phi_n$ as a subset of the $d$-dimensional tensor product of $1$-dimensional polynomial basis with maximum total degree $M$:
$$
\Phi^{TD}_n:=\left\{\{\phi_i(x)\}_{i=1}^n\mathrel{\bigg|}\phi_i(x)=\prod_{j=1}^d \phi_i^j\left(x_j\right)=\prod_{j=1}^d x_j^{r_i^j}, \;\sum_{j=1}^d r_i^j \leq M\right\}.
$$
The separability of the basis reduces the computational complexity of assembling the  Galerkin residual equations \cref{Eq:residual}. However, the use of a total degree basis only partially circumvents the curse of dimensionality, as the cardinality of $\Phi^{TD}_n$ grows combinatorial with $M$ and $d$, limiting its applicability to $d\lesssim20$. Alternatively, we also consider a \emph{hyperbolic cross polynomial basis} \cite{azmi2021optimal}, defined by
$$
\Phi^{HC}_n:=\left\{\{\phi_i(x)\}_{i=1}^n\mathrel{\bigg|}\phi_i=\prod_{j=1}^d \phi_i^j\left(x_j\right)=\prod_{j=1}^d x_j^{r_i^j}, \;\prod_{j=1}^d\left(r_i^j+1\right) \leq J+1\right\},
$$
where $J$ is the maximum degree of the basis.
Compared to the  polynomial basis truncated by total degree, the hyperbolic cross basis scales better for high-dimensional problems.  More specifically, in this paper we include numerical results up to $d=30$, with the maximum degree of basis functions $4$. 
In such case, the full multidimensional basis would contain $1.52\times 10^{18}$ elements, the basis truncated by total degree $M=4$ contain $46374$ elements, while the hyperbolic cross basis with  $J=4$ has only $555$ elements. 

\subsection*{High-dimensional integration}
Recall that given $u_n^{(m)}=-\frac{1}{\epsilon} \nabla\Phi_n^\top \mathbf{c}_n^{(m)}$ at the $m$-th iteration, we solve the GHJB equation 
\begin{equation*}
\left\langle-\mu V_n^{(m+1)}+\left(D V_n^{(m+1)}\right)^\top u_n^{(m)}+f+\frac{\epsilon}{2}|u_n^{(m)}|^2, \Phi_n\right\rangle=\mathbf{0}_n
\end{equation*}
for $V_n^{(m+1)}=\Phi_n^\top \mathbf{c}_n^{(m+1)}$, which leads to a linear system for $\mathbf{c}_n^{(m+1)}$ depending on $\mathbf{c}_n^{(m)}$:
$$
\left(-\mu\mathbf{M}+\mathbf{G}\left(\mathbf{c}_n^{(m)}\right)\right) \mathbf{c}_n^{(m+1)}=-\mathbf{F}-\mathbf{L}(\mathbf{c}_n^{(m)}).
$$
Note that $\mathbf{0}_n$ denotes a  vector consisting $n$ entries, all of which are zero. 
 We follow the steps in \cite[section~3.3,4.2]{Kalise_2018} to expand the different terms in the GHJB equation. Note that 
$$
\left\langle -\mu V_n^{(m+1)}, \Phi_n\right\rangle=-\mathbf{M}\, \mathbf{c}_n^{(m+1)}, 
$$
where $\mathbf{M} \in \mathbb{R}^{n \times n},$ with $\mathbf{M}_{(i, j)}=\mu\left\langle\phi_i, \phi_j\right\rangle =\mu\prod_{p=1}^d \int_{\Omega_p} \phi_i^p\left(x_p\right) \phi_j^p\left(x_p\right) d x_p$. Then
$$
\begin{aligned}
\left\langle \left(D V_n^{(m+1)}\right)^\top  u_n^{(m)}, \Phi_n\right\rangle & =\mathbf{G}\left(\mathbf{c}_n^{(m)}\right) \mathbf{c}_n^{(m+1)}, \quad \mathbf{G} \in \mathbb{R}^{n \times n}, \\
\mathbf{G}_{(i, j)}\left(\mathbf{c}_n^{(m)}\right)
&:=-\frac{1}{\epsilon} \sum_{k=1}^n c_k^{(m)}\sum_{p=1}^d\tilde{\mathbf{U}}_{(i,j,k,p)},
\end{aligned}
$$
where  $\tilde{\mathbf{U}} \in \mathbb{R}^{n \times n \times n \times d }$ is given by
\begin{equation*}
\begin{aligned}
&\tilde{\mathbf{U}}_{(i,j,k,p)}:=\left\langle(\partial_{x_p} \phi_k)( \partial_{x_p} \phi_j), \phi_i\right\rangle\\
& =\left(\prod_{\substack{q=1\\ 
q \neq p}}^d \int_{\Omega_q} \phi_i^q(x_q)\phi_j^q(x_q) \phi_k^q(x_q) d x_q\right)\left(\int_{\Omega_p} \phi_i^p(x_p)  \left(\partial_{x_p}\phi_j^p(x_p) \right)\left(\partial_{x_p} \phi_k^p(x_p)\right) d x_p\right) .
\end{aligned}
\end{equation*}
For $\left\langle\frac{\epsilon}{2}|u_n^{(m)}|^2, \Phi_n\right\rangle$ note that
$$
\begin{aligned}
\frac{\epsilon}{2}|u_n^{(m)}|^2
&=\frac{1}{2\epsilon}\sum_{p=1}^d\left(\sum_{j=1}^n c_j^{(m)}\partial_{x_p}\phi_j\right)\left(\sum_{k=1}^n c_k^{(m)}\partial_{x_p} \phi_k\right)
\end{aligned}
$$
leading to
$$
\left\langle\frac{\epsilon}{2}|u_n^{(m)}|^2, \Phi_n\right\rangle=\frac{1}{2\epsilon} \left(\mathbf{c}^{(m)}\right)^\top  \sum_{p=1}^d\tilde{\mathbf{U}}_{(\bullet,p)}\mathbf{c}^{(m)}:=\mathbf{L}(\mathbf{c}^{(m)}).
$$
Finally, we  discuss the computation of $\left\langle f(x), \Phi_n\right\rangle$ separately for the case of the objective function $f:\mathbb{R}^d\rightarrow\mathbb{R}$ being separable or not.
If $f$ is separable, we suppose there exists a tensor-valued function $\mathcal{F}:\mathbb{R}^d\rightarrow\mathbb{R}^{n_f\times d}$ such that
$$f(x)=\sum_{j=1}^{n_f} \prod_{p=1}^d \mathcal{F}_{( j, p)}\left(x_p\right),$$
where $n_f\in\mathbb{N}$ is called separation rank. 
Then, $$
\left\langle f(x), \Phi_n\right\rangle=\sum_{j=1}^{n_f}\left\langle\left(\prod_{p=1}^d \mathcal{F}_{(j, p)}(x_p)\right), \Phi_n\right\rangle=:\mathbf{F},
$$
where
$\mathbf{F}_{(i)}=\sum_{j=1}^{n_f}\left(\prod_{\substack{p=1}}^d \int_{\Omega_p} \mathcal{F}_{(j, p)}(x_p) \phi_i^p \left(x_p\right) d x_p\right).
$ 

If the function $f$ is non-separable, we can apply a direct Monte Carlo method using $N_{mc}\in\mathbb{N}$ uniform samples $\{\bar{x}_q\}_{q=1}^{N_{mc}}$. Then $$\mathbf{F}_{(i)}:=\left\langle f(x), \phi_i\right\rangle=|\Omega|\frac{1}{N_{mc}}\sum_{q=1}^{N_{mc}}f(\bar{x}_q)\phi_i(\bar{x}_q),$$
 where $\bar{x}_q=(\bar{x}_{q,1},\ldots,\bar{x}_{q,d})\in\mathbb{R}^d$ and $\phi_i(\bar{x}_q)=\prod_{p=1}^d\phi_i^p(\bar{x}_{q,p})$. Although this is a one-time, offline computation,  more advanced Monte Carlo methods could be employed to improve accuracy or efficiency, e.g.,  Importance Sampling \cite{tokdar2010importance}, Sequential Monte Carlo (SMC) \cite{del2006sequential}, or Markov Chain Monte Carlo (MCMC) \cite{geyer1992practical}. However, in this work we opted for the simplest possible approach to maintain clarity and focus on the core contribution.
\section{The controlled-CBO methodology}
\label{control_CBO}
Having computed an optimal feedback law approximately steering a single particle towards the global minimizer of $f$, we now turn our attention to the controlled-CBO algorithm given by the dynamics \cref{eq:cbo_intro}. In \cite{totzeck2020consensusbased} the CBO methodology  is enhanced 
by switching on or off the drift part of the SDE according to whether the particle has a lower function value compared with the particle mean. We apply this idea to the feedback control term and propose a system of $N \in \mathbb{N}$ interacting particles $\{X_t^i\}_{i=1}^N$, which evolves in time with respect to a system of SDEs given by
  \begin{equation}
d X_t^i=[-\tilde{\lambda}(t, X_t^i)\left(X_t^i-v_\alpha(\rho_t^N)\right)+\tilde{\beta}(t, X_t^i) u^*(X_t^i)]d t+\sigma \operatorname{Diag}(X_t^i-v_\alpha(\rho_t^N)) d W_t^i,   
\label{SDE_controlled}
\end{equation}
$$
\begin{aligned}
& \tilde{\lambda}(t, X_t^i)=\lambda\mathrm{H}\left(f(X^i_t)-f(v_\alpha(\rho_t^N))\right), \\
& \tilde{\beta}(t, X_t^i)=\beta\mathrm{H}\left(f(X^i_t)-f^{approx}(X^i_t)\right),
\end{aligned}
$$
\begin{equation*}
    f^{approx}(x)=\sum_{j=1}^n a_j\phi_j(x), 
\end{equation*}
with $H(x):= \begin{cases}1, & x \geq 0 \\ 0, & x<0\end{cases}$ being the Heaviside function, $\lambda>0$ and $\beta>0$ constant parameters. The polynomial approximation $f^{approx}$ is a projection of the function $f$ onto $\Phi_n$. Let   $\boldsymbol{b}=\left(\left\langle\phi_1, f\right\rangle, \left\langle\phi_2, f\right\rangle, \ldots, \left\langle\phi_n, f\right\rangle\right)^{\top}$ and define the mass matrix $G\in\mathbb{R}^{n\times n}$ by 
$G_{p,q}=\left\langle\phi_p, \phi_q\right\rangle$, for $p,q=1,\cdots,n $. Then the coefficient vector $\mathbf{a}=\left(a_1, a_2, \ldots, a_n\right)^{\top}$ is determined by the linear system $G\boldsymbol{a}=\boldsymbol{b}.$
In the special case $\Phi_n$ is an orthogonal basis, the  matrix $G$ is diagonal, and the coefficients simplify to $a_j=\frac{\langle\phi_j(x),f(x)\rangle}{\langle\phi_j(x),\phi_j(x)\rangle}$, $j=1,\cdots,n$.
 We assume that the initial condition of the particles $X_0^i \in \mathbb{R}^d$  are independent and identically distributed with $\operatorname{law}\left(X_0^i\right)=\rho_0 \in \mathcal{P}\left(\mathbb{R}^d\right)$.  The consensus point $v_\alpha(\rho_t^N)$ is calculated as before by the weighted average presented in \eqref{v_f}-\eqref{weight}.

The first term in drift of the model \cref{SDE_controlled} directs the particle system towards a  consensus point. If the objective function value at weighted average location $v_\alpha(\rho_t^N)$ is lower than the function value at $X_t^i$, 
the particle moves towards the consensus point $v_\alpha(\rho_t^N)$ and the strength is given by the distance $\lambda\left|X_t^i-v_\alpha(\rho_t^N)\right|$.
The second component in the drift term uses information directly from the objective function. It  
provides a possible descent towards a
global minimum of the approximated objective function. The particles compare the original objective function value $f(\cdot)$ and approximated function  value $f^{approx}(\cdot)$ at current position $X_t^i$, if the approximated   value $f^{approx}(X_t^i)$ is lower,  the control term is active;  otherwise, the system reduce to CBO. We compare with the polynomial approximation to $f$, since $u^*$ is computed on this basis. As the number of elements $n$ allowed in the basis $\Phi_n$ increases, we expect a corresponding enhancement in 
 the accuracy of the polynomial approximation. Consequently, the  control $u^*$ is expected to have higher accuracy in finding the descent towards the true global minimizer of the objective function. As before, the diffusion term introduces randomness to explore the search space. Note that the standard CBO is recovered by setting
$
\tilde{\lambda}(t, X) \equiv \lambda$ and $\tilde{\beta}(t, X) \equiv 0 .
$  

The success of CBO methods in addressing non-convex and non-smooth optimization problems stems from the fact that it can be interpreted as stochastic relaxations of gradient-based methods, but relying solely on evaluations of the  objective function \cite{riedl2023gradient}. As the number of optimizing particles tends to infinity, the CBO implements a convexification of a rich class of functions and reveals a similar behaviour to stochastic gradient decent \cite{fornasier2024consensus}. 
Since the control signal is derived from the gradient of the value function associated with a control problem depending on the original objective, it can be viewed as a ``gradient-like''  global information of the objective function, which is independent 
of the size of particle system and its initialization. Even when the initial positions of the particles lie in unfavorable regions, the control can steer them toward better regions where the CBO dynamics can be more effective and robust with respect to the choice of parameters such as $\sigma$ and $N$. 

\subsection*{Well-Posedness of the controlled-CBO dynamics and mean-field behaviour}
\label{wellpose}
In this section, we extend the well-posedness results in \cite{carrillo2018analytical} to a simplified version of \cref{SDE_controlled}, i.e. the dynamics in \eqref{eq:cbo_intro} as well as the associated Fokker-Planck equation.
In our proposed methodology for the controlled-CBO model \cref{SDE_controlled}, the Heaviside functions imposed on $\tilde{\beta}$ and $\tilde{\lambda}$ aim to circumvent issues arising from insufficient accuracy of the HJB approximation and
mitigate concentration in local minimal points, respectively. For simplicity, we analyse the case without Heaviside functions. We also note that even in practise the numerical performance of these simplified controlled dynamics improve performance relative to standard CBO. 

Before stating the result, let us introduce notations. The space of continuous functions $f: \mathcal{X} \rightarrow \mathcal{Y}$ is denoted by $\mathcal{C}(\mathcal{X}, \mathcal{Y})$, where $\mathcal{X} \subset \mathbb{R}^n$, $n\in\mathbb{N}$ and 
$\mathcal{Y}$ is a suitable  topological space. Further,  $\mathcal{C}_c^k(X, Y)$ and $\mathcal{C}_b^k(X, Y)$ 
denote continuous function spaces in which  functions  are $k$-times continuously differentiable, have compact support and are bounded, respectively. The main objects of study are laws of stochastic processes $\rho \in \mathcal{C}\left([0, T], \mathcal{P}\left(\mathbb{R}^d\right)\right)$. A fixed measure $\varrho \in \mathcal{P}\left(\mathbb{R}^d\right)$  is said to belong to $\mathcal{P}_p\left(\mathbb{R}^d\right)$, for $1 \leq p<\infty$, if it has finite $p$-th moment $\int |x|^p d \varrho(x)$.

We will consider the mean field dynamics of the particle system in \eqref{eq:cbo_intro} and the corresponding nonlinear Markov process when all particles are initialized independently with $\text{law}(X_0^i)=\rho_0$. 
As $N$ increases we expect the particles to cover more densely the state space and the method to perform better. From a theoretical standpoint, it is important to characterize and establish the well-posedness of the limiting mean-field regime.
When $N\rightarrow\infty$, propagation of chaos dictates that processes $X_t^i$ will behave as independent,  and we may obtain a limiting consensus point
$$
v_\alpha( \rho_t^N):=\frac{1}{\sum_{i=1}^N\omega_f^\alpha\left(X_t^i\right)}\sum_{i=1}^N X_t^i \omega_f^\alpha\left(X_t^i\right) \rightarrow \frac{1}{\int_{\mathbb{R}^d} \omega_f^\alpha d \rho_t} \int_{\mathbb{R}^d} x\; \omega_f^\alpha d \rho_t:=v_\alpha(\rho_t)\,,
$$
see \cite{oelschlager1984martingale}. Here, $\rho_t \in \mathcal{P}\left(\mathbb{R}^d\right)$ is a Borel probability measure describing the single particle
distribution resulting from the mean-field limit, which is assumed to be absolutely continuous with respect to the Lebesgue measure $d x$. 

We begin by defining weak solutions of the (nonlinear) Fokker-Planck equation.
\begin{definition}
\label{weak_def}
For fixed parameters $\lambda,\beta,\sigma>0.$
 We say $\rho \in \mathcal{C}\left([0, T], \mathcal{P}\left(\mathbb{R}^d\right)\right)$ satisfies the following Fokker Planck equation 
$
\partial_t \rho_t= \nabla\cdot\left[\left(\lambda (x-v_\alpha(\rho_t))-\beta u^*(x)\right) \rho_t\right]\\+\frac{\sigma^2}{2}\sum_{k=1}^d \partial_{k k}\left(\operatorname{Diag}\left(x-v_\alpha\left(\rho_t\right)\right)_{k k}^2 \rho_t\right)
$
with initial condition $\rho_0 \in \mathcal{P}\left(\mathbb{R}^d\right)$ in the weak sense, if we have for all $\phi \in \mathcal{C}_c^{\infty}\left(\mathbb{R}^d\right)$ and all $t \in(0, T)$

\begin{align}
\frac{d}{d t} \int \phi(x) d \rho_t(x)&=-\lambda \int \langle x-v_\alpha\left(\rho_t\right), \nabla \phi(x)\rangle d \rho_t(x)+\beta\int \langle u^*(x), \nabla \phi(x)\rangle d \rho_t(x)\nonumber\\ &+\frac{\sigma^2}{2} \int \sum_{k=1}^d \operatorname{Diag}\left(x-v_\alpha\left(\rho_t\right)\right)_{k k}^2 \partial_{k k}^2 \phi(x)d \rho_t(x)
\label{FP}
\end{align}
and $\lim _{t \rightarrow 0} \rho_t=\rho_0$ pointwise.

\end{definition}
Throughout the section, we assume that $f$ and $u^*$ satisfy the following assumptions:
\begin{assumption}
\begin{enumerate}
    \item Assume the objective function is bounded from below and
there exists $x^* \in \mathbb{R}^d$ such that $f\left(x^*\right)=\min _{x \in \mathbb{R}^d} f(x)=: \underline{f}$.
\item  There exist constants $L_f,\,c_f>0$ such that
$$
\begin{cases}|f(x)-f(y)| \leq L_f(|x|+|y|)|x-y| & \text { for all } x, y \in \mathbb{R}^d \\ f(x)-\underline{f} \leq c_f\left(1+|x|^2\right) & \text { for all } x \in \mathbb{R}^d\end{cases}
$$
\item There exists a constant $L_u>0$ such that $|u^*(x)-u^*(y)| \leq L_u|x-y|$ for all $ x, y \in \mathbb{R}^d$.
\end{enumerate}
\label{assumption}
\end{assumption}
\cref{assumption} is standard for the mean-field analysis of CBO methods, see \cite{carrillo2018analytical,fornasier2024consensus}. The global Lipschitz continuity of the control signal $u^*$ directly implies a linear growth in $x$, i.e., there exists a constant $c_u>0$ such that for all $x\in \mathbb{R}^d$,
$|u^*(x)|^2 \leq c_u\left(1+|x|^2\right)$. While this condition may not be trivial in control theory, in practice, it can be easily checked by computing the gradient of the basis functions $\{\frac{\partial \phi_i}{\partial x}\}_i$. 
Also, this assumption is satisfied by a broad class of non-convex objective functions whose approximated value function over a compact domain is quadratic.  In the supplementary material \ref{appendix: qua} we verify \cref{assumption} for $f$ being a quadratic function. 
%\end{remark}

The main result of this section is provided by the following theorem.
\begin{theorem}
Let $f$ and $u^*$ satisfy \cref{assumption}, and $\rho_0 \in \mathcal{P}_4\left(\mathbb{R}^d\right)$. Then, there exists a unique nonlinear process $\bar{X} \in \mathcal{C}\left([0, T], \mathbb{R}^d\right), T>0$, satisfying
\begin{equation}
d \bar{X}_t=[-\lambda\left(\bar{X}_t-v_\alpha(\rho_t)\right)+\beta u^*(x)] d t+\sigma \operatorname{Diag}\left(\bar{X}_t-v_\alpha(\rho_t)\right) d W_t,
\label{sdee}
\end{equation}
$\rho_t=\operatorname{law}\left(\bar{X}_t\right)$ in the strong sense, and $\rho \in \mathcal{C}\left([0, T], \mathcal{P}_4\left(\mathbb{R}^d\right)\right)$ satisfies the corresponding Fokker Planck equation  in the sense of \cref{weak_def}.
\label{thm}
\end{theorem}
\begin{proof}
    See \cref{app_1}.
\end{proof}
As we indicated earlier, as $N\rightarrow\infty$ each $X_t^i$ in \eqref{eq:cbo_intro} is expected to become independent (with $i$) and converge weakly to the dynamics of $\bar{X}_t$ in \eqref{sdee}, see \cite{oelschlager1984martingale} for the case $\beta\equiv 0$. The well-posedness of \eqref{sdee} rigorously characterizes the dynamics of the best case scenario that uses infinite of particles, which is impossible in practice, but close to using a very large number of particles.

\section{Numerical experiments}
\label{numerical2}
In this section, we assess the convergence of the particle dynamics toward the point of uniform consensus at $x_*$ as well as the overall performance of our controlled-CBO algorithm compared to the standard CBO implementation. As performance measures we use the $2$-Wasserstein distance to a Dirac delta $\delta_{x_*}$:
\begin{equation}
W_2^2\left(\rho_t^N, \delta_{x_*}\right)=\int|x-x_*|^2d \rho_t^N(x)\quad \text { for } \quad \rho_t^N=\frac{1}{N} \sum_{i=1}^N \delta_{X_t^i},
   \label{distance} 
\end{equation}
and the particle system variance, which is computed according to $$\operatorname{Var}\left(\rho_t^N\right):=\frac{1}{2} \int\left|x-\mathbb{E}\left(\rho_t^N\right)\right|^2 d \rho_t^N(x), \quad\mathbb{E}\left(\rho_t^N\right)=\int x \,d\rho_t^N(x). $$
As benchmark problems we will consider
 the $d$-dimensional Ackley function \cite{ackley2012connectionist}:$$f_A(\mathbf{x})=-20 \exp \left(-0.2 \sqrt{\frac{1}{d} \sum_{i=1}^d x_i^2}\right)-\exp \left(\frac{1}{d} \sum_{i=1}^d \cos \left(2\pi x_i\right)\right)+21+\exp (1)$$
and the $d$-dimensional Rastrigin function \cite{Rastrigin}:
$$f_R(\mathbf{x})=10 (d+1)+\sum_{i=1}^d\left[x_i^2-10 \cos \left(2 \pi x_i\right)\right].$$
Both benchmarks have global minimizer at $\mathbf{x}^*=(0, \ldots, 0)$ with $f_A\left(\mathbf{x}^*\right)=1$ and $f_R\left(\mathbf{x}^*\right)=10$. For each benchmark, we start by selecting a basis  $\Phi_n$ (with  cardinality $n$) and set
$
\Omega=[-2,2]^d,\,\epsilon=0.1, \;\mu=0.1,\;$ and $\theta=0.5$ to obtain the approximate optimal feedback law $u^*$. Then, we conduct numerical experiments of  the controlled-CBO model \cref{SDE_controlled}.
Recall that the polynomial approximation of the objective function $f^{approx}$ is computed by projecting $f$ onto the basis $\Phi_n$. 
The particle dynamics \cref{SDE_controlled} are discretized using the Euler–Maruyama scheme.
In the following, the parameters for the particle simulation are chosen as
$
d t=10^{-1}, \;\alpha=40,\;\sigma=0.7,\;\beta=1,\;\lambda=1,\;T=10.
$
The particle simulation is stopped once the final time $T =10$ is reached.
If $\tilde{\beta}\equiv 0$, the setting \cref{SDE_controlled} recovers the standard CBO.  Regarding the choice of polynomial basis, although the overall performance of the feedback control is generally insensitive to the basis selection \cite{azmi2021optimal, kalise2020robust, Kalise_2018}, high-degree expansions using a full monomial basis may suffer from ill-conditioning of mass matrix, potentially affecting numerical stability and accuracy. In our context, however, 
the approximation of value function is not intended to recover every subtle feature of the objective. Instead, a low-resolution surrogate, computed at moderate cost, is sufficient to indicate a favorable region near the true solution, allowing the associated feedback control to guide the CBO dynamics effectively. For these reasons,  both monomial and Legendre bases yield comparably effective results in the numerical tests  presented below. \textsc{MATLAB} codes for controlled CBO and the numerical tests presented in this paper are available in the GitHub repository: \url{https://github.com/AmberYuyangHuang/ControlledCBO}.

\subsection{2-dimensional benchmark problems}
\label{sec:2dnumerical}
In this section for each benchmark problem we consider a favourable and a non-favourable initialization distribution of the particle system separately, where in one case the  global minimum is contained in the initial distribution support and in the other it is not.

\subsubsection{Ackley function}
We consider monomial basis with  total degree  $M=4$ for approximating the value function.
First, we consider an advantageous initialization where the particle system are initially equidistantly distributed on $[-1,0.5]^2$. This includes the global minimizer of the Ackley function.

Figure \ref{1a} shows the evolution of $\operatorname{Var}\left(\rho_t^N\right)$ and $W_2^2(\rho_t^N,\delta_{x^*})$ for a different number of particles $N$ with initialization $\rho_0=\mathcal{U}[-1,0.5]^2$ in the standard CBO model. To obtain a visualization of how the particle system evolves,  Figure \ref{1b} depicts the trajectory of $N=50$ particles. The numerical results of the controlled-CBO are shown in Figure \ref{1c} and \ref{1d}. It can be seen that if the initialized distribution of particle system encloses the global minimizer, both algorithms exhibit  convergence to the global minimizer. Nevertheless, the controlled-CBO is notably faster and achieves a higher level of precision.
\begin{figure}[h!]
\centering
\captionsetup[subfigure]{aboveskip=-1pt,belowskip=-1pt}
  \begin{subfigure}[t]{0.45\textwidth}
\includegraphics[width=\textwidth]{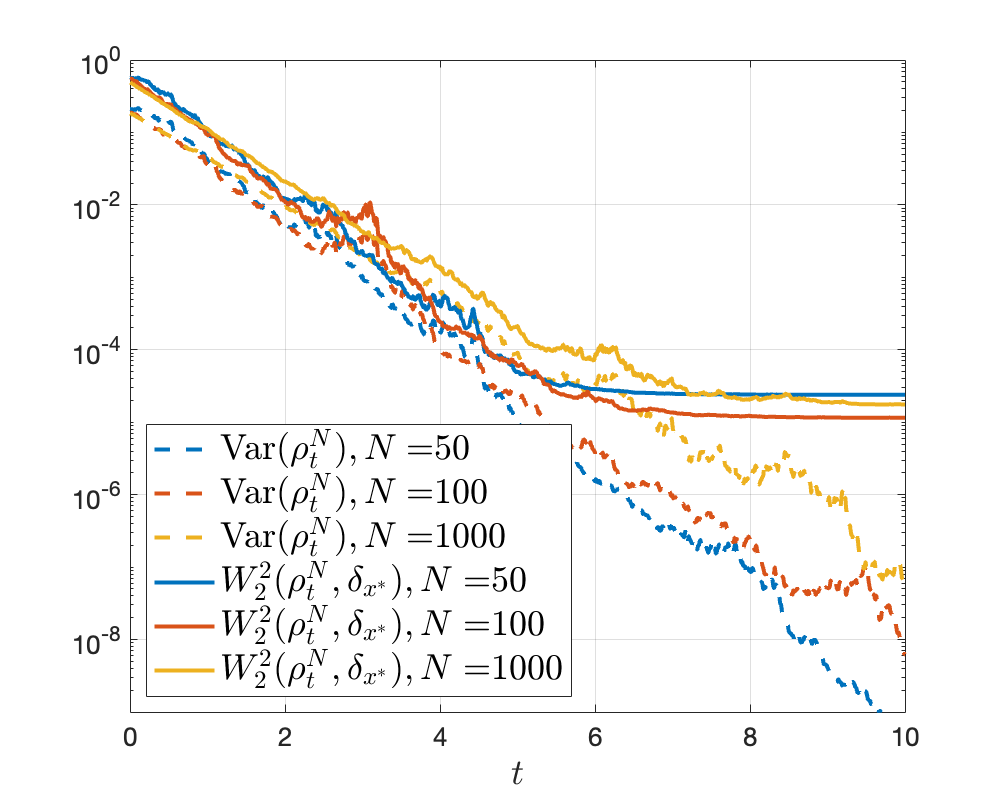}
  \caption{Evolution of the variance $\operatorname{Var}\left(\rho_t^N\right)$ and $W_2^2(\rho_t^N,\delta_{x^*})$ in standard CBO for different number of particles $N$}
  \label{1a}
  \end{subfigure}
\begin{subfigure}[t]
  {0.49\textwidth}
\includegraphics[width=\textwidth]{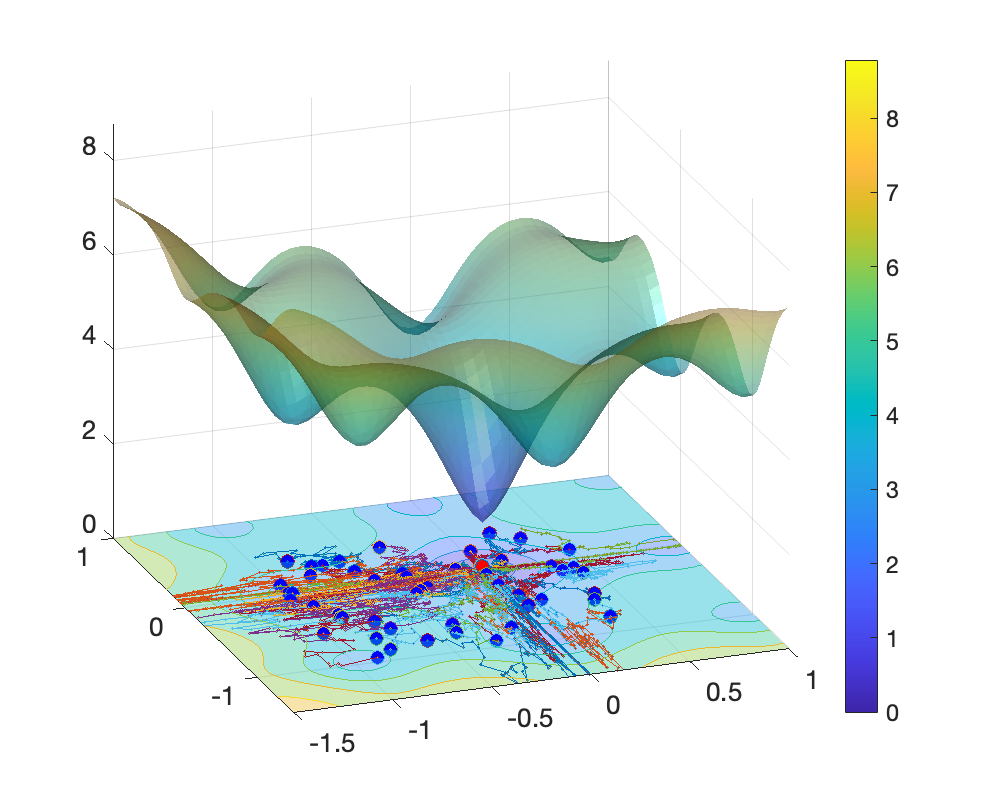}
  \caption{Trajectory of $N=50$ particles under the standard CBO, blue points are the initial position of particles and red points are the final position of particles at time $T=10.$}
  \label{1b}
  \end{subfigure}
\begin{subfigure}[b]
  {0.45\textwidth}
\includegraphics[width=\textwidth]{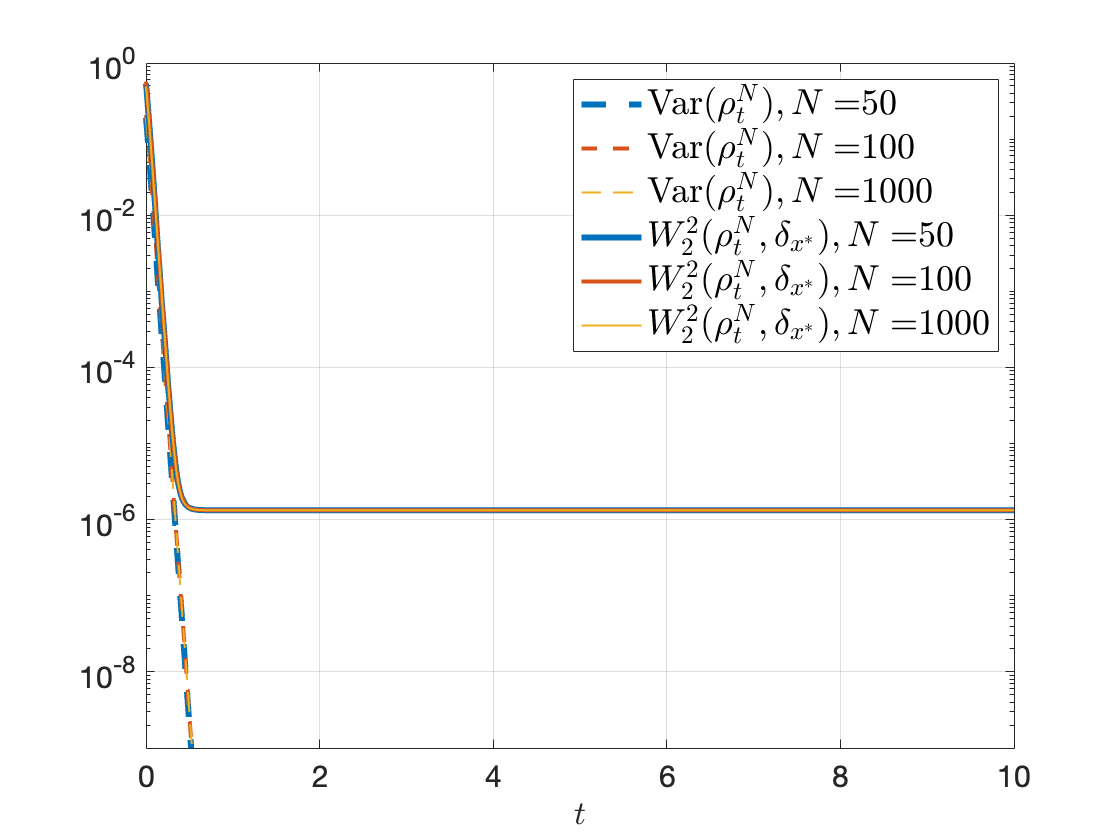}
\caption{Evolution of the variance $\operatorname{Var}\left(\rho_t^N\right)$ and $W_2^2(\rho_t^N,\delta_{x^*})$ in controlled-CBO.}
\label{1c}
\end{subfigure}
\begin{subfigure}[b]
  {0.49\textwidth}
\includegraphics[width=\textwidth]{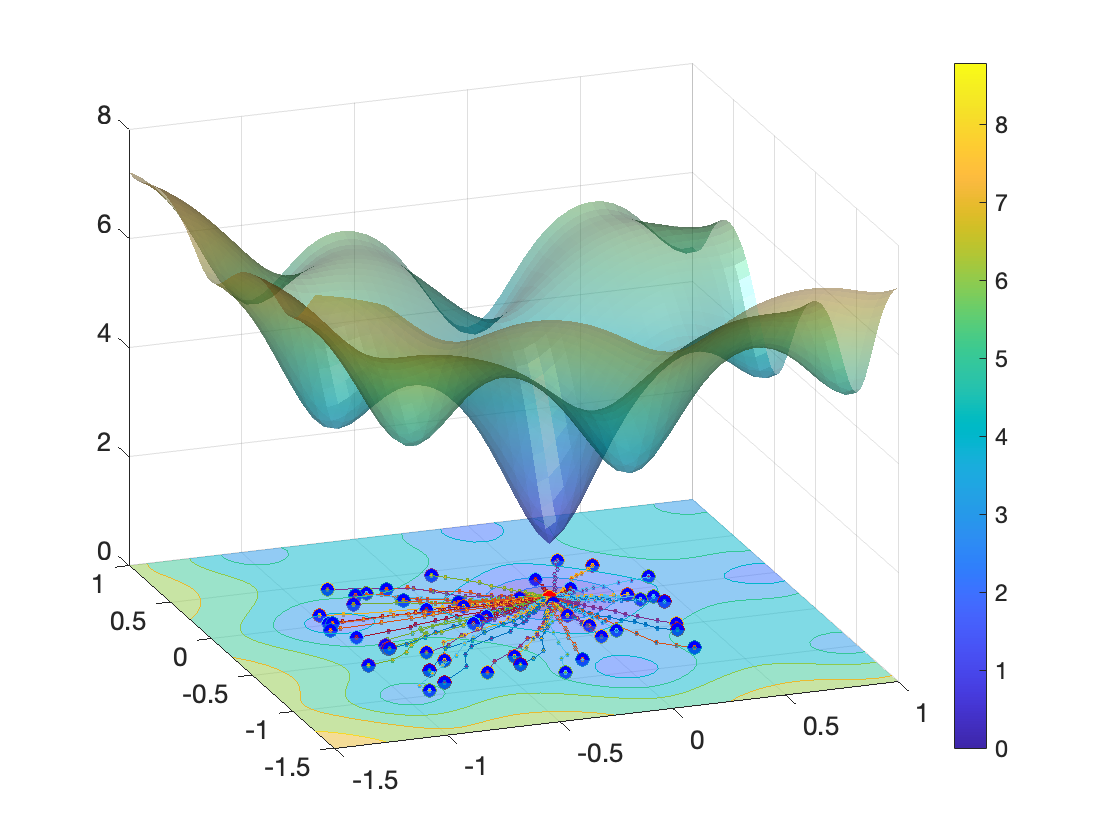}   
 \caption{Trajectory of $N=50$ particles under the controlled-CBO}
 \label{1d}
\end{subfigure}
\caption{The comparison between standard CBO and controlled-CBO in 2-dimensional Ackley function with initialization $\rho_0=\mathcal{U}[-1,0.5]^2$.  Both methods obtain convergence; however,
the controlled CBO demonstrates faster convergence towards the global minimizer and achieves higher accuracy. As shown in sub-figure (d), all particles follow a direct path from their initial positions to  $x^*$.}
\label{A1}
\end{figure}
Secondly, we consider an initial configuration in which particles are  equidistantly distributed within $[-1, -0.5]^2$. Figure \ref{A2} reveals that the controlled-CBO exhibits superior robustness in comparison to the standard CBO. This enhancement is due to the control term, which directs the particles in the direction of the global minimizer, whereas the particles under the standard CBO
tend to move towards local minimizers.
\begin{figure}[h!]
\centering
\captionsetup[subfigure]{aboveskip=-1pt,belowskip=-1pt}
  \begin{subfigure}[t]{0.45\textwidth}
\includegraphics[width=\textwidth]{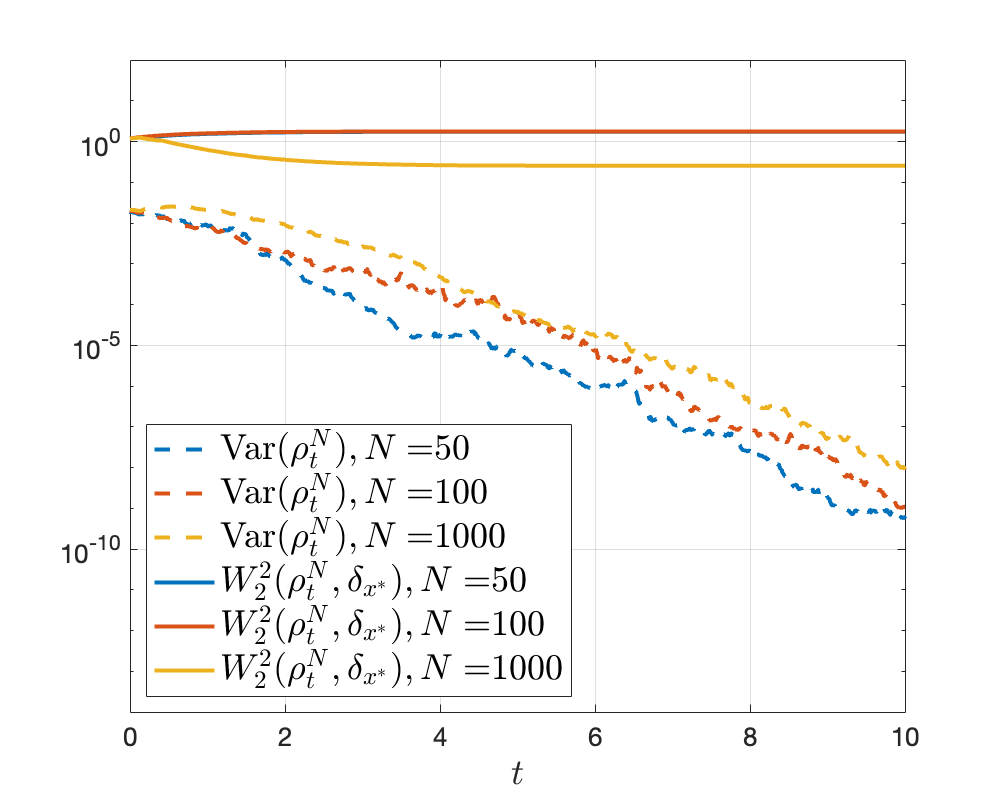}
  \caption{ Evolution of the variance $\operatorname{Var}\left(\rho_t^N\right)$ and $W_2^2(\rho_t^N,\delta_{x^*})$ in standard CBO}
  \label{2a}
  \end{subfigure}
\begin{subfigure}[t]
  {0.49\textwidth}
\includegraphics[width=\textwidth]{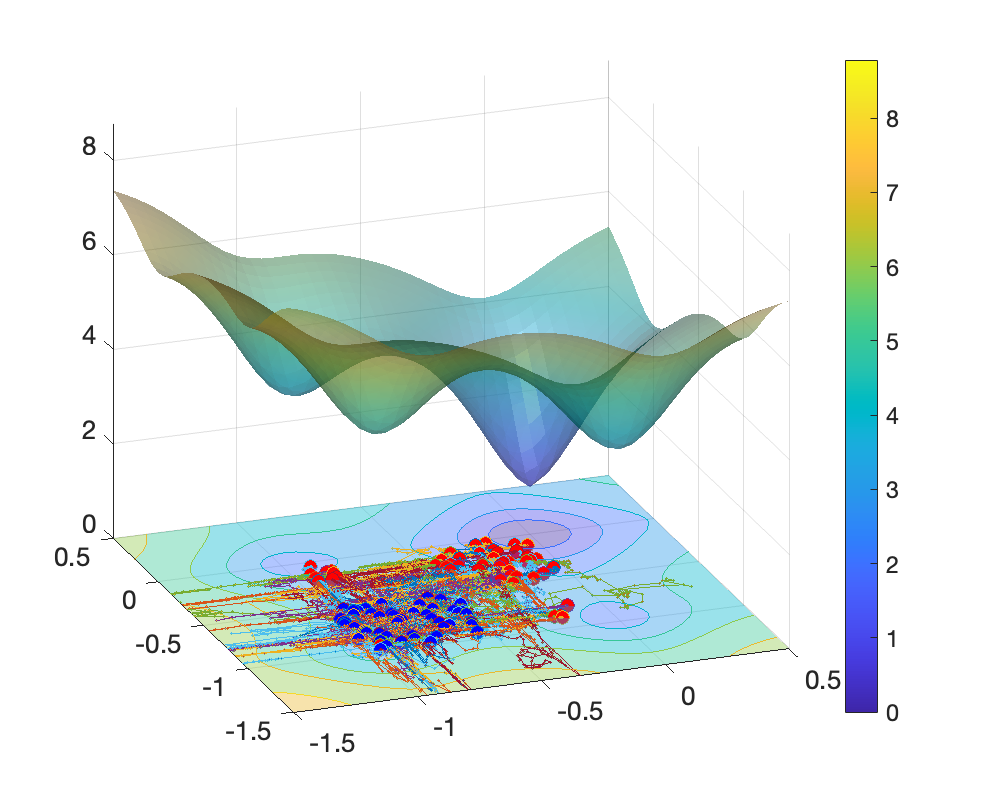}
  \caption{Trajectory of $N=50$ particles under the standard CBO}
   \label{2b}
  \end{subfigure}
\begin{subfigure}[t]
  {0.45\textwidth}
\includegraphics[width=1\textwidth]{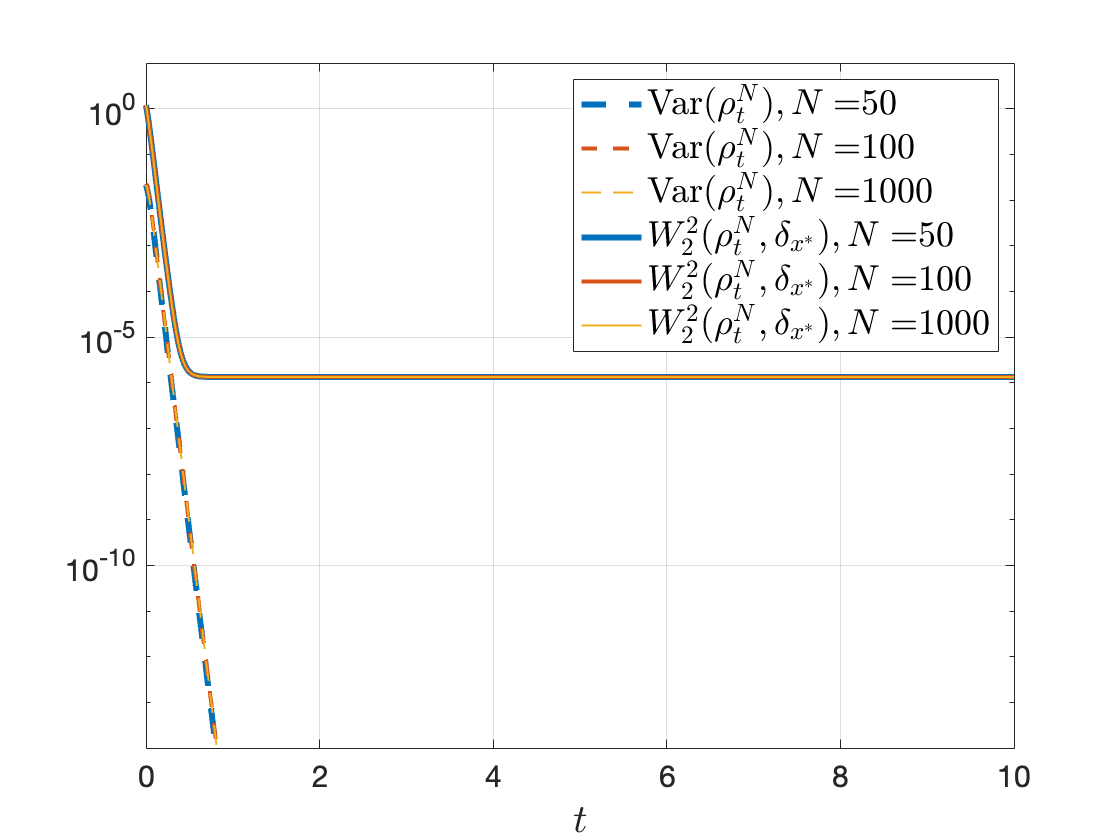}
\caption{ Evolution of the variance $\operatorname{Var}\left(\rho_t^N\right)$ and $W_2^2(\rho_t^N,\delta_{x^*})$ in controlled-CBO}
 \label{2c}
\end{subfigure}
\begin{subfigure}[t]
  {0.49\textwidth}
\includegraphics[width=\textwidth]{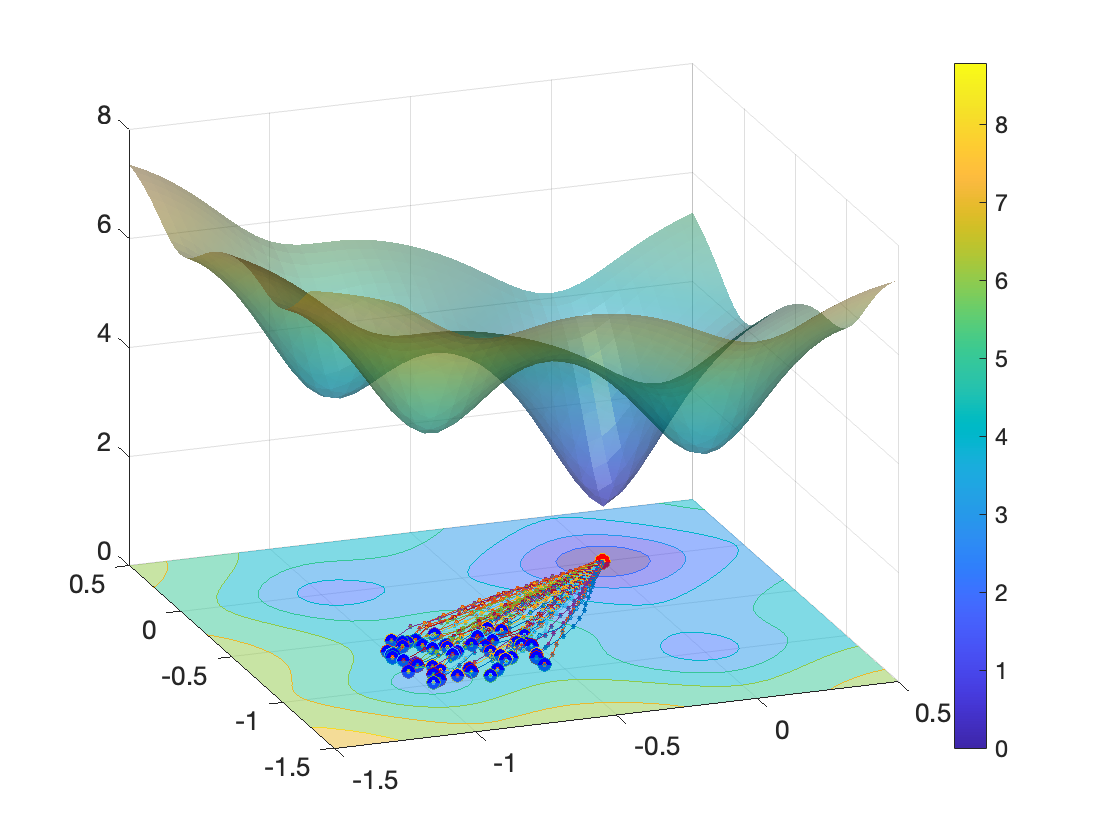}   
 \caption{Trajectory of $N=50$ particles under the controlled-CBO}
  \label{2d}
\end{subfigure}
\caption{The comparison between CBO and controlled-CBO in 2-dimensional  Ackley function with initialization of $\rho_0=\mathcal{U}[-1,-0.5]^2$. The standard CBO with non-favourable initialization fails to obtain convergence. While most of particles concentrate near the global minimizer, some tend to move toward local minimizers. In contrast, controlled CBO maintains  fast convergence and high accuracy, with particles  move almost directly to $x^*$.}
\label{A2}
\vspace{-0.2cm}
\end{figure}
To further illustrate how the value function and the associated feedback control guide the optimization process, \Cref{ack_V} presents plots of the approximated value function and control field derived from its gradient for the  Ackley function.
\begin{figure}[h]
\begin{subfigure}[t]
  {0.49\textwidth}
\includegraphics[width=\textwidth]{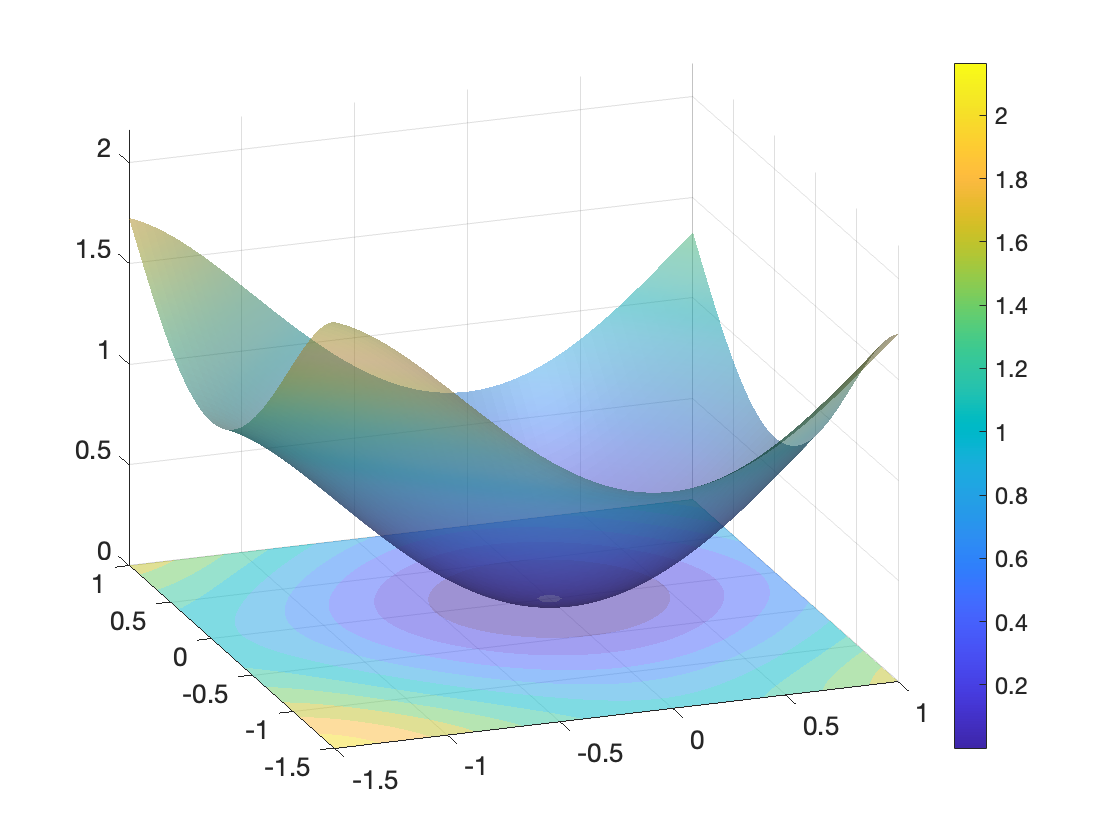}   
 \caption{Approximated value function}
\end{subfigure}
\begin{subfigure}[t]
  {0.49\textwidth}
\includegraphics[width=\textwidth]{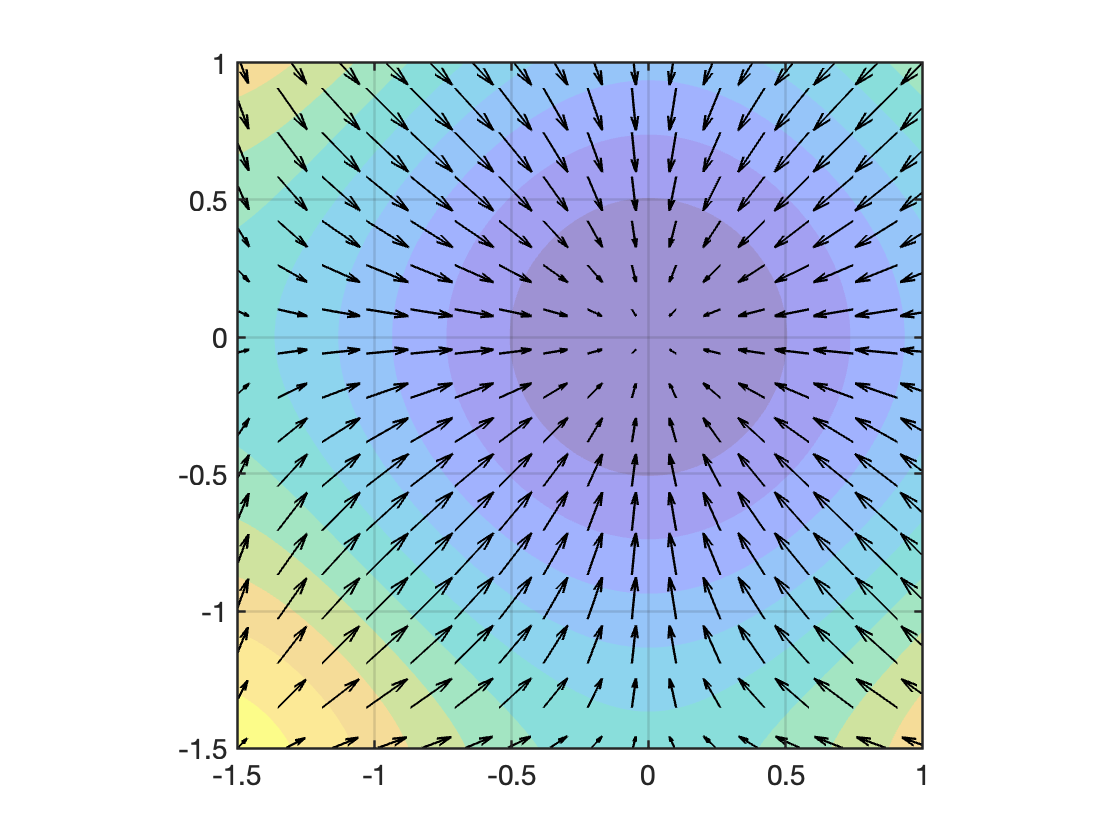}   
 \caption{Feedback control field}
\end{subfigure}
  \caption{Illustration of the approximated value function and the associated feedback control field, which guides the particle system toward the global minimizer for the Ackley function.}
  \label{ack_V}
  \end{figure}
  
\subsubsection{Rastrigin function}
We consider a similar numerical experiment as for the Ackley function, using a Legendre polynomial basis with  
total degree  $M=4$. 
\begin{figure}[h!]
\centering
\captionsetup[subfigure]{aboveskip=-1pt,belowskip=-1pt}
  \begin{subfigure}[t]{0.45\textwidth}
\includegraphics[width=\textwidth]{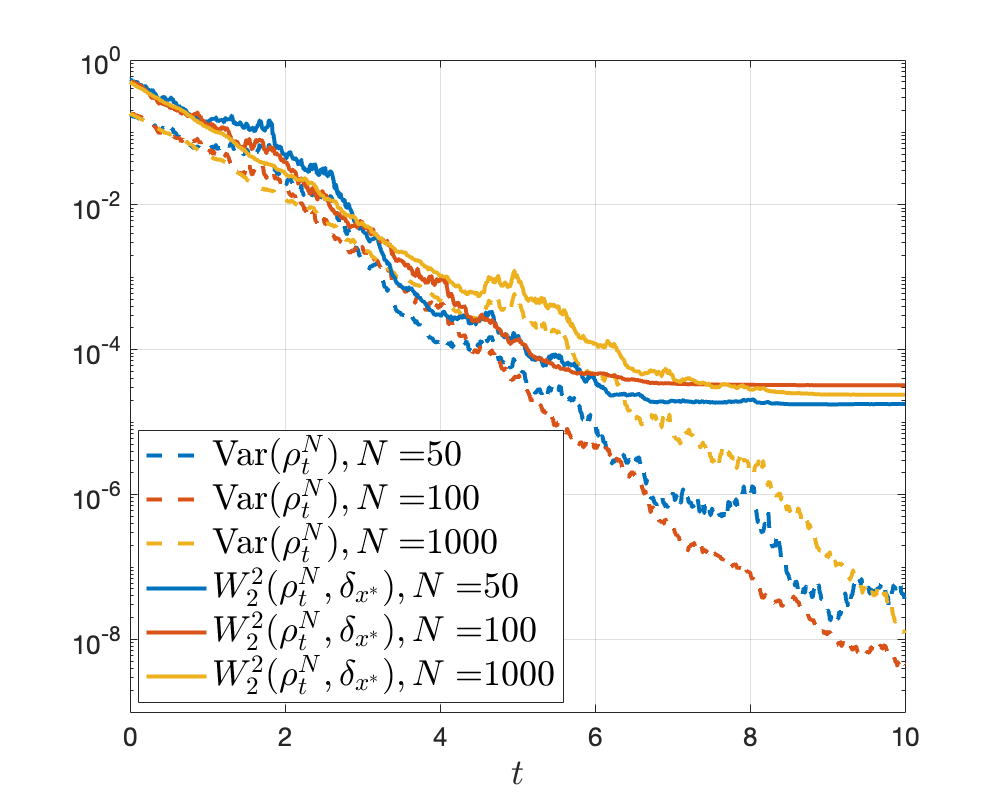}
  \caption{ Evolution of the variance $\operatorname{Var}\left(\rho_t^N\right)$ and $W_2^2(\rho_t^N,\delta_{x^*})$ in standard CBO}
  \label{r1}
  \end{subfigure}
\begin{subfigure}[t]
  {0.49\textwidth}
\includegraphics[width=\textwidth]{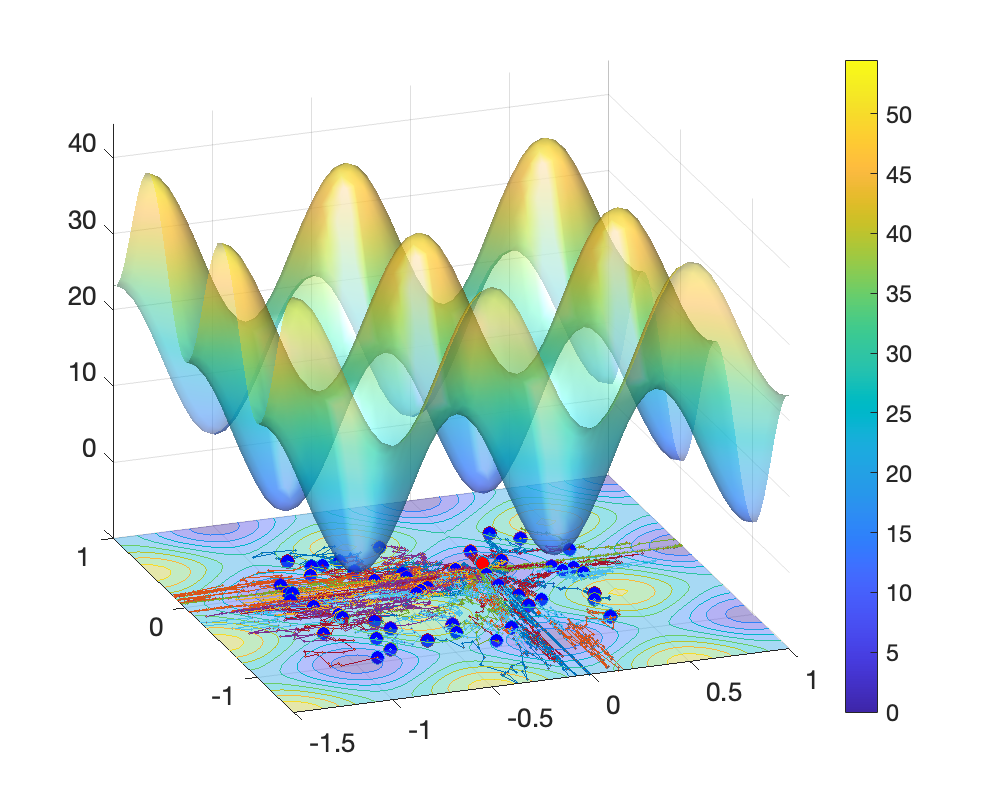}
  \caption{Trajectory of $N=50$ particles under the standard CBO}
   \label{r2}
   \end{subfigure}
\begin{subfigure}[b]
  {0.45\textwidth}
\includegraphics[width=1\textwidth]{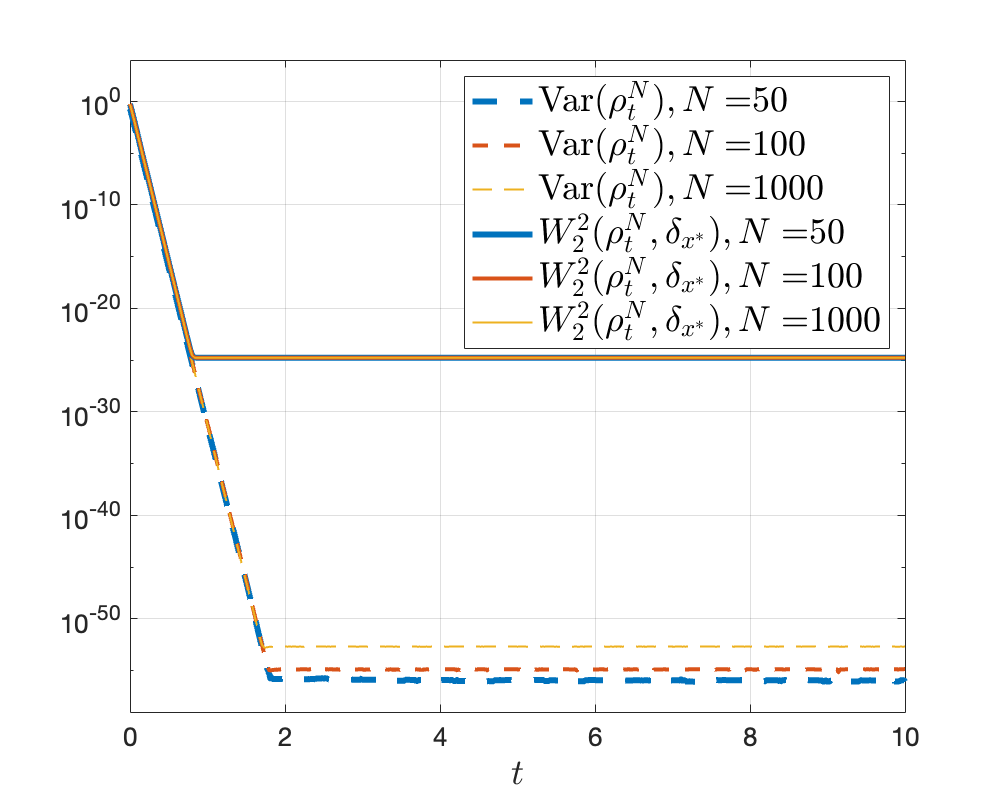}
\caption{Evolution of the variance $\operatorname{Var}\left(\rho_t^N\right)$ and $W_2^2(\rho_t^N,\delta_{x^*})$ in controlled-CBO}
 \label{r3}
\end{subfigure}
\begin{subfigure}[b]
  {0.49\textwidth}
\includegraphics[width=\textwidth]{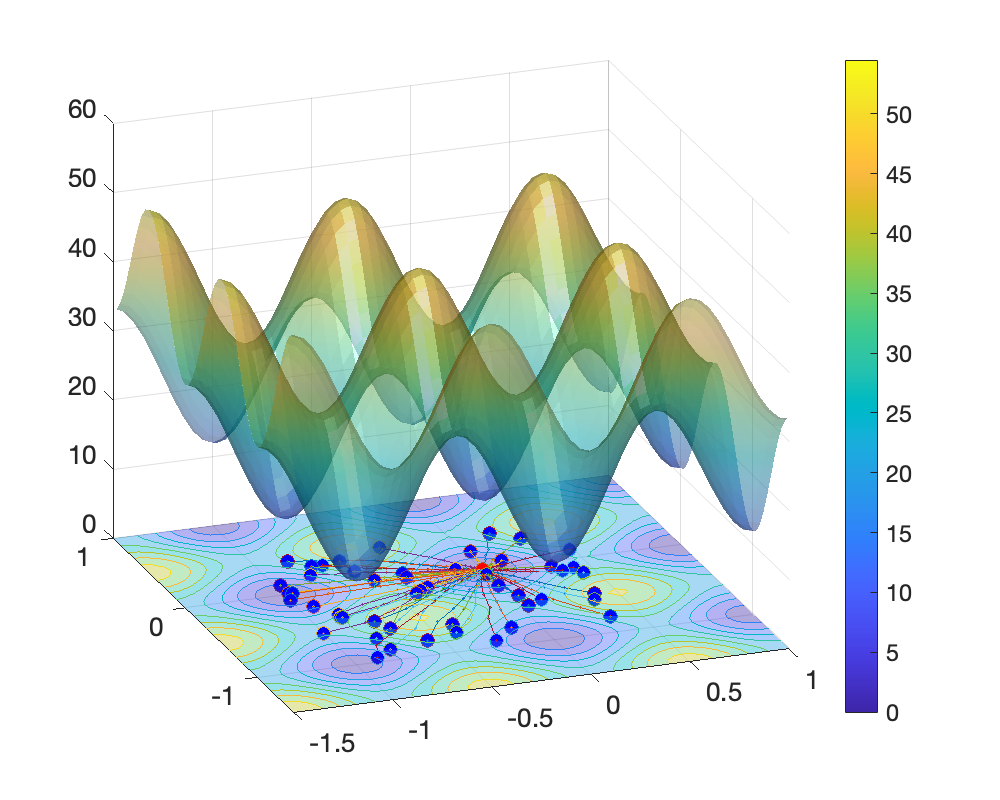}   
 \caption{Trajectory of $N=50$ particles under the controlled-CBO}
  \label{r4}
\end{subfigure}
\caption{The comparison between CBO and controlled-CBO in 2-dimensional  Rastrigin function with initialization of $\rho_0=\mathcal{U}[-1,0.5]^2$. While both methods achieve convergence, controlled CBO shows superior convergence speed and accuracy.}
\label{R1}
\vspace{-0.15cm}
\end{figure}
The results are presented in \cref{R1}. The figures show that if the initial mass distribution of the particle system encloses the global minimizer of the Rastrigin function, both standard CBO and controlled-CBO converge. As before, the controlled-CBO method exhibits significantly
faster convergence and higher level of accuracy. It can also be seen from the trajectory in \cref{r4} that the particles move directly to the global minimizer.
Then, we consider particles are initially equidistantly distributed on $[-1,-0.5]^2$,  which  does not contain the global minimizer $x^*=(0,0)$.
\begin{figure}[h!]
\centering
\captionsetup[subfigure]{aboveskip=-1pt,belowskip=-1pt}
  \begin{subfigure}[t]{0.45\textwidth}
\includegraphics[width=\textwidth]{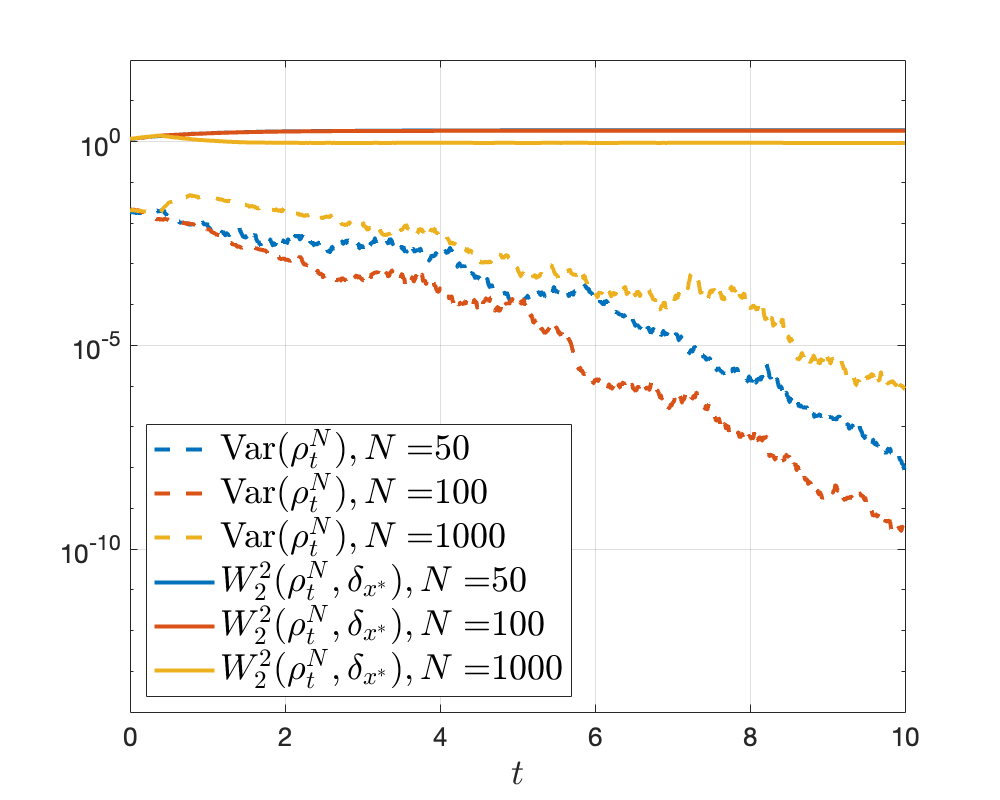}
  \caption{ Evolution of the variance $\operatorname{Var}\left(\rho_t^N\right)$ and $W_2^2(\rho_t^N,\delta_{x^*})$ in standard CBO}
  \end{subfigure}
\begin{subfigure}[t]
  {0.49\textwidth}
\includegraphics[width=\textwidth]{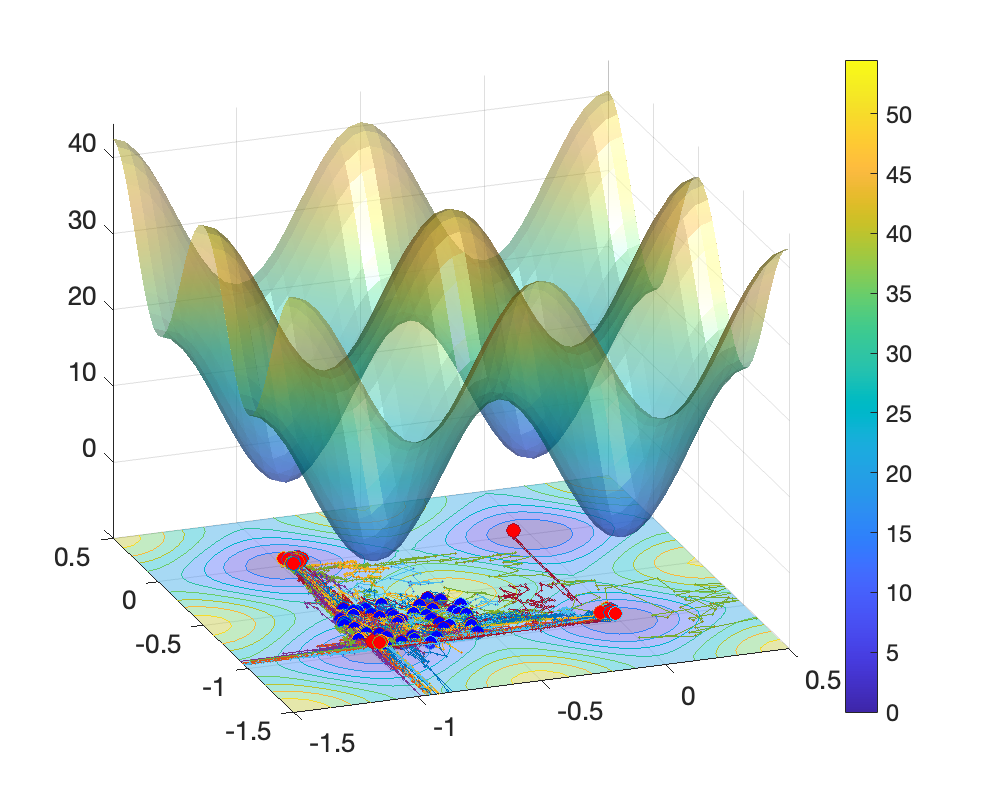}
  \caption{Trajectory of $N=50$ particles under the standard CBO}
  \end{subfigure}
\begin{subfigure}[b]
  {0.45\textwidth}
\includegraphics[width=1\textwidth]{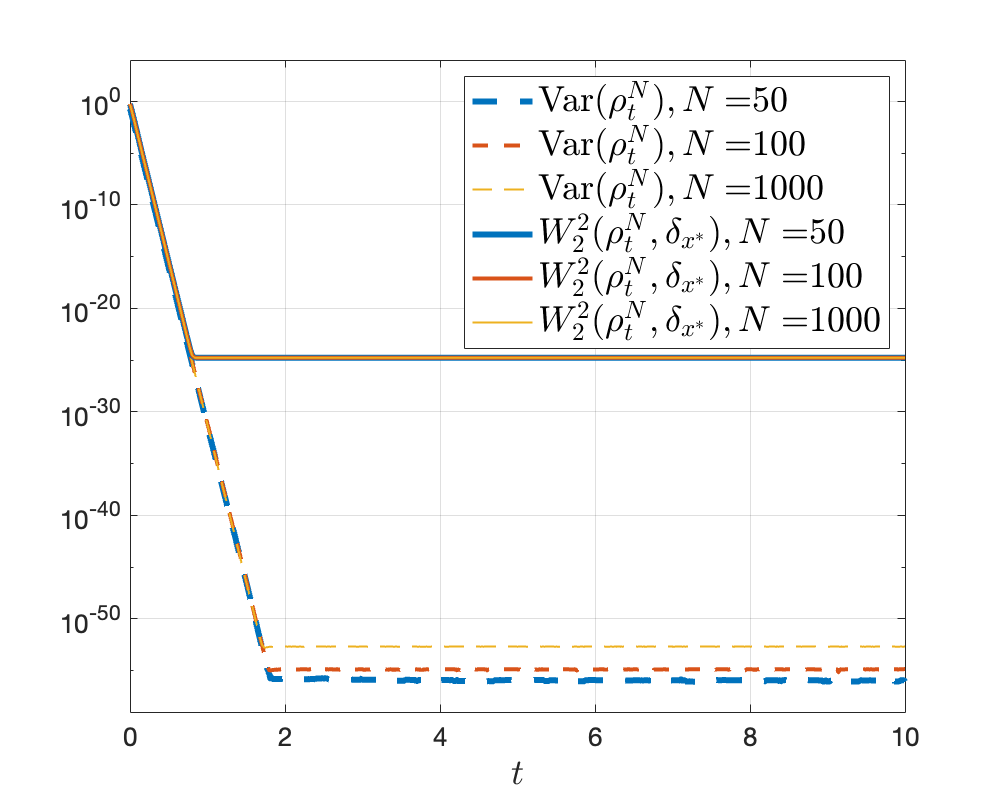}
\caption{ Evolution of the variance $\operatorname{Var}\left(\rho_t^N\right)$ and $W_2^2(\rho_t^N,\delta_{x^*})$ in controlled-CBO}
\end{subfigure}
\begin{subfigure}[b]
  {0.49\textwidth}
\includegraphics[width=\textwidth]{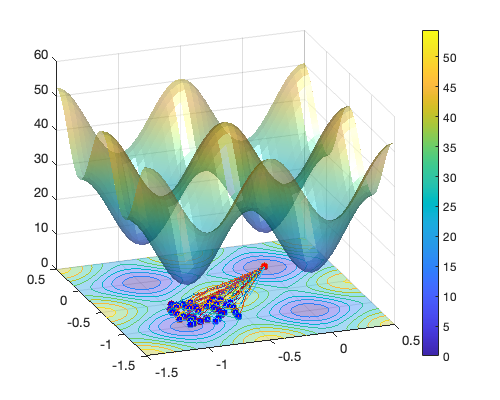}   
 \caption{Trajectory of $N=50$ particles under the controlled-CBO}
\end{subfigure}
\caption{The comparison between CBO and controlled-CBO in 2-dimensional  Rastrigin function with initialization $\rho_0=\mathcal{U}[-1,-0.5]^2$. The sub-figures (b) and (d) further reveal that some particles driven by standard CBO become trapped in local minimizers, whereas the controlled particles follow a direct path toward the global minimizer $x^*$.
}
\label{R2}
\end{figure}
Figure \ref{R2} provides additional empirical support for our conclusion. Under the standard CBO algorithm, some particles tend to stay in local minimizers and, within our simulation setup and time frame, fail to converge towards the global minimizer.
Conversely, the controlled-CBO algorithm exhibits superior performance due to the application of an optimal control signal. This effect is further illustrated in \Cref{ras_V}, which shows the approximated value function and the associated control field, clearly directing the particles toward the  global minimizer.
\begin{figure}[h!]
\begin{subfigure}[t]
  {0.49\textwidth}
\includegraphics[width=\textwidth]{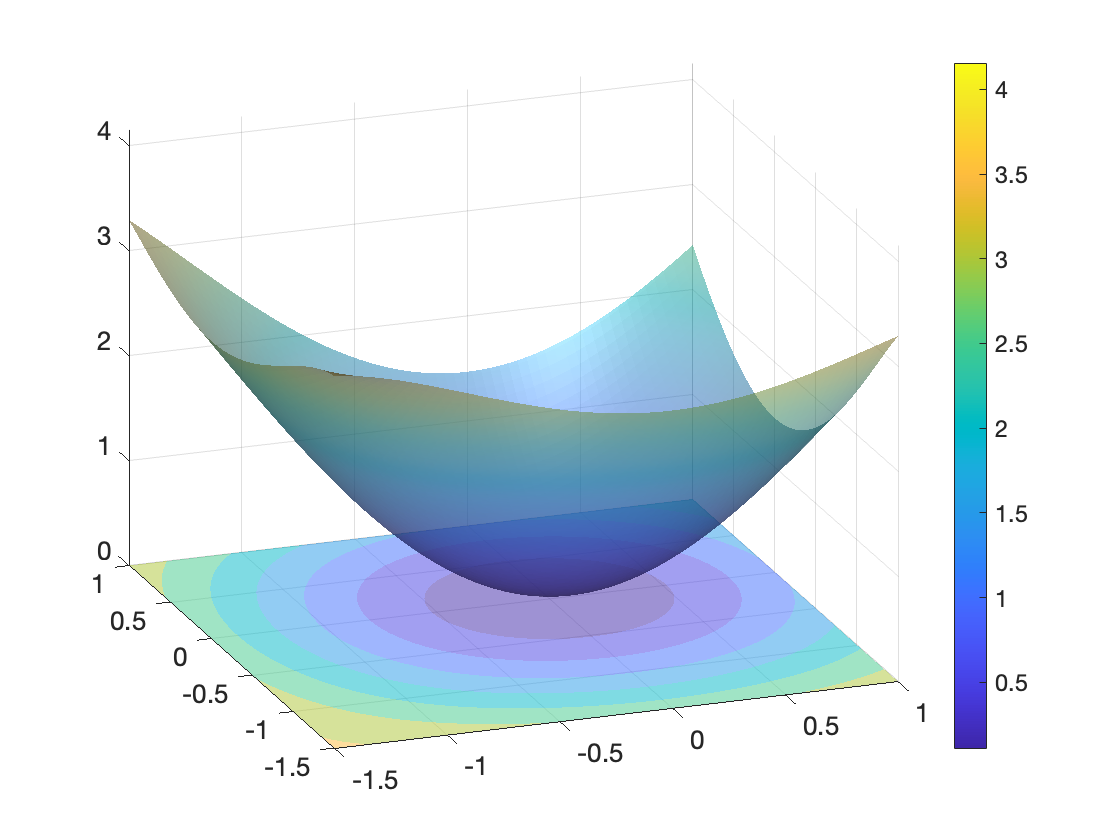}   
 \caption{Approximated value function}
\end{subfigure}
\begin{subfigure}[t]
  {0.49\textwidth}
\includegraphics[width=\textwidth]{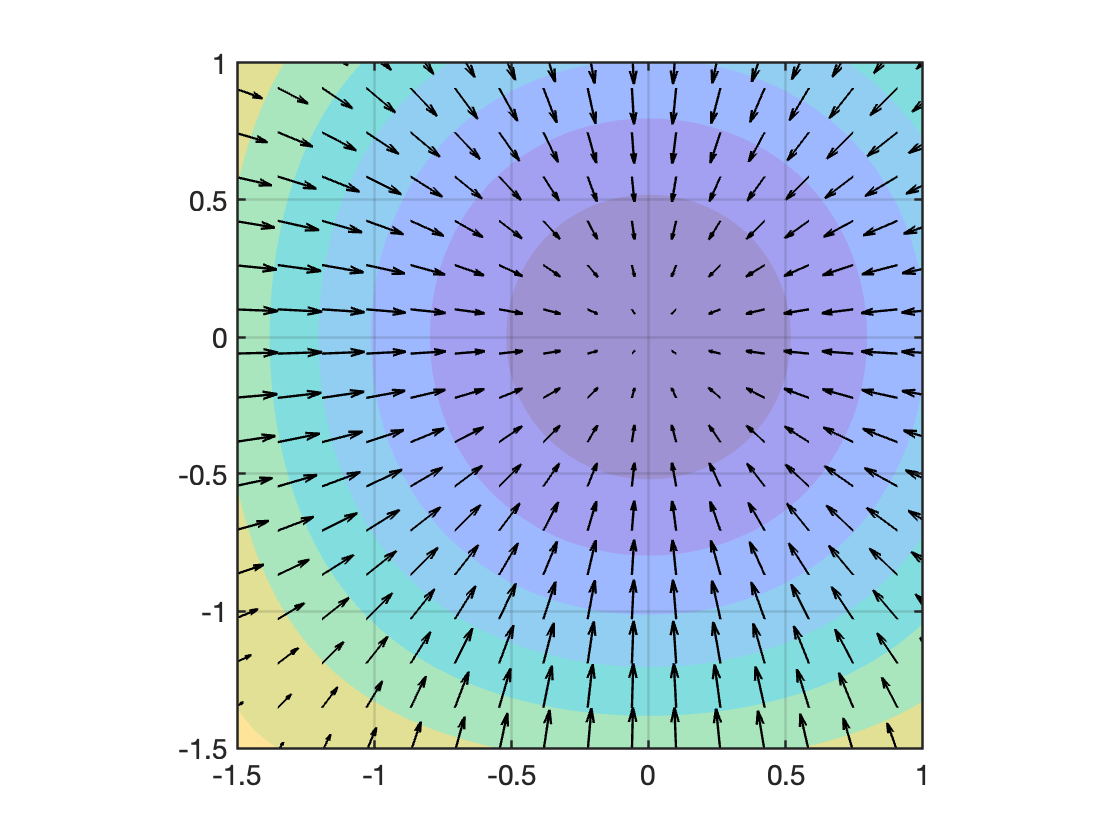}   
 \caption{Feedback control field}
\end{subfigure}
  \caption{Approximated value function and control field for the Rastrigin function. The feedback control steers particles toward the global minimizer.}
  \label{ras_V}
  \end{figure}

\subsection{Increasing the degree of HJB Approximation}
\label{degree}
The effectiveness of the controlled-CBO algorithm depends on 
the accuracy of the approximation of solution to the HJB equation.
By allowing higher-degree polynomials to appear in the basis, we can enhance the accuracy of the HJB approximation, subsequently improving the performance of the controlled-CBO method. To illustrate this idea, we begin by providing an example. Consider an $1$-dimensional objective function
\begin{equation}
f (x) = (x^2-2.2)^2 - 0.08x + 0.5, 
\label{2_min}
\end{equation}
whose global minimum is
located at $x^*= 1.48776$ with
$f(x^*)=0.38116$ and a local minimum located at $x=-1.47867$ with value $0.618477$. For this problem, we apply Legendre polynomial basis truncated by total degree $M$ and implement the \cref{alg2} over the bounded domain $[-4,4]$. \Cref{2min_value} reveals that when the degree of approximation is relatively low (e.g., $M=2,\;4,\;6$), the approximation fails to identify the global minimizer, possibly reducing the effectiveness of the controlled-CBO method.  When the degree of approximation increases to $M=8$, the approximation is more accurate and 
the resulting control $u^*$  is able to indicate a correct direction towards the area around the global minimizer $x^*=1.48776$. Afterwards, the CBO mechanism refines the search and reaches $x^*$ more precisely. Accordingly, even if the initial configuration of the particle system is far away from the global minimizer, the controlled-CBO converges successfully, see \cref{2min_com}.  To further illustrate how the control term influences the particle dynamics, \Cref{2min_tra} displays trajectories of the uncontrolled gradient flow $\dot{x}=-\nabla f(x)$ and the controlled flow $\dot{x}= u_n(x)$, both implemented by Euler–Maruyama scheme over time horizon $T=10$ with step size $dt=0.01$.  Starting from the initial position $x(0)=-2,$ the gradient flow does not escape the local minimum, whereas the controlled dynamics gets closer to the true global minimizer as $M$ increases. 

\begin{figure}[h!]
  \subfloat[]{
    \includegraphics[width=0.44\textwidth]{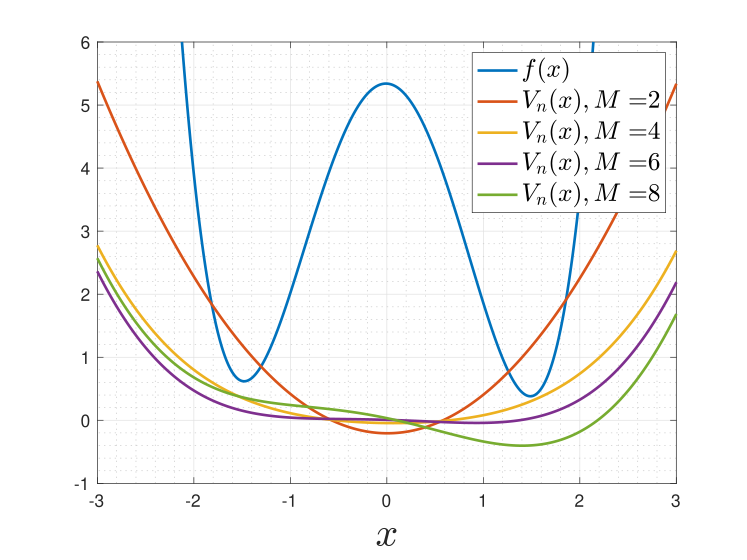}
    \label{2min_value}
  }
  \subfloat[]{
    \includegraphics[width=0.45\textwidth]{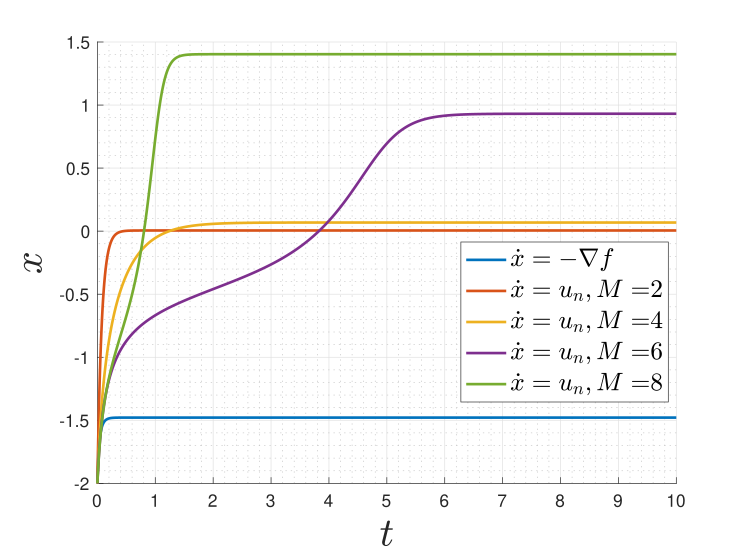}
     \label{2min_tra}
  }
\caption{(a) The comparison between the exact objective function $f$ in \eqref{2_min} and the approximated value function $V_n$ with different total degree $M$. As $M$ increases, $V_n$ begins to recover the global minimizer of $f$ and magnifies the discrepancy between local and global minimizers.   (b) Trajectories of the gradient flow $\dot{x}=-\nabla f(x)$ and  the controlled flows $\dot{x}=u_n(x)$, all initialized at $x(0)=-2$. For sufficiently large $M$, the trajectories escape the local minimum and converge to the true global minimizer.}
%\vspace{-0.5cm}
\end{figure}
\begin{figure}[h!]
\centering
\captionsetup[subfigure]{aboveskip=-1pt,belowskip=-1pt}
\begin{subfigure}{0.45\textwidth}
\includegraphics[width=\textwidth]{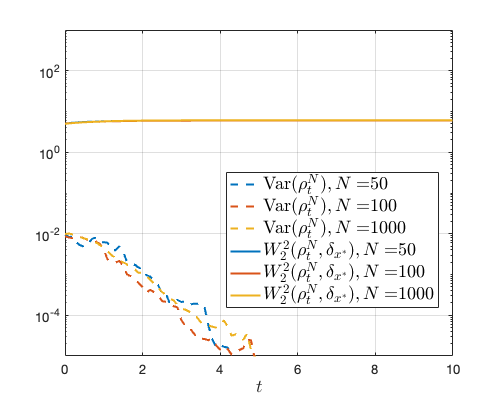}
\caption{Standard CBO }
\end{subfigure}
\begin{subfigure}{0.45\textwidth}
\includegraphics[width=\textwidth]{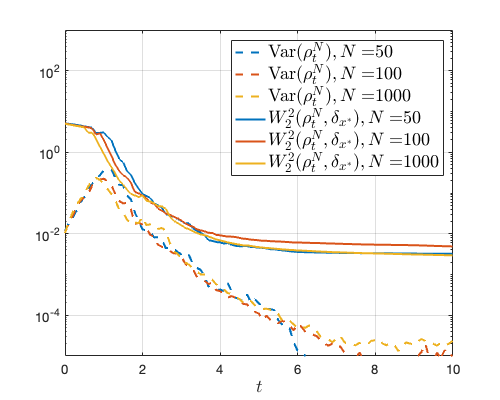}   
 \caption{Controlled-CBO}
\end{subfigure}
\caption{The evolution of the variance $\operatorname{Var}\left(\rho_t^N\right)$ and $W_2^2(\rho_t^N,\delta_{x^*})$ of standard CBO and controlled-CBO in minimizing the function \cref{2_min} with  initialization $\rho_0=\mathcal{U}[-1,-0.5]$, which does not encircle the global minimizer. The controlled-CBO achieves convergence, whereas the standard CBO fails to converge.}
\label{2min_com}
\vspace{-0.15cm}
\end{figure}
It is worth noting that this problem is intentionally designed to require higher-degree polynomial expansions for accuracy, making it particularly challenging. In such cases, the
use of  Legendre polynomials  provides more stable and accurate resolution at a similar computational cost. However, in our context,  monomial polynomial basis can perform satisfactorily, even in non-smooth settings. To show this, we provide a numerical test for a non-smooth function
\begin{equation}
g(x)= \begin{cases}x^2, & x<-2 \\ 4, & -2 \leq x \leq 0 \\ 4(x-1)^2 & x>0\end{cases}, 
\label{objective_nonsmooth}
\end{equation}
whose global minimum is located at $x^*=1$. For this problem, we apply monomial polynomial basis truncated by total degree $M$ and implement the \cref{alg2} over the bounded domain $[-3,3]$. The comparison between the original objective function and the approximated value functions is shown in \Cref{nonsmooth}. This test problem highlights the advantage of using $DV$ over $Df$ guiding the dynamics toward the global minimizer, as discussed in \Cref{rem:1}. The flat region on the interval $[-2,0]$ poses a significant challenge for gradient-based methods, whereas the value function $V$ serves not only a smooth surrogate for the objective,  but also provides more informative guidance in locating global minimizer, see \Cref{nonsmooth_tra}.
\begin{figure}[h!]
\centering
  \subfloat[]{
    \includegraphics[width=0.44\textwidth]{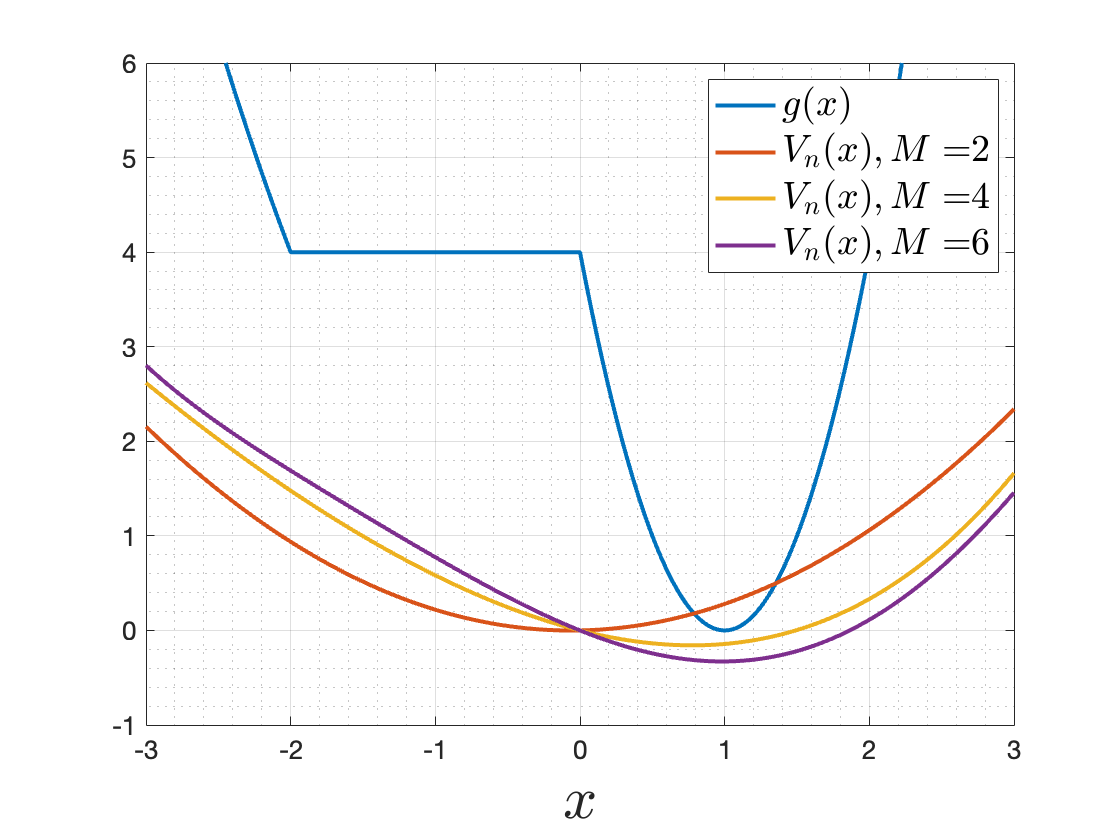}
    \label{nonsmooth}  }
  \subfloat[]{
    \includegraphics[width=0.45\textwidth]{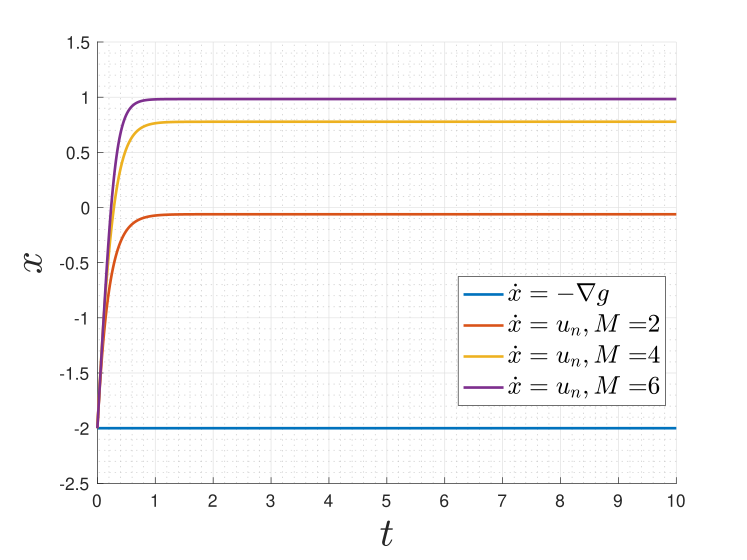}
     \label{nonsmooth_tra}
  }
\caption{(a) The comparison between the exact objective function $g$ in \eqref{objective_nonsmooth} and the approximated value function $V_n$ with different total degree $M$. (b) Trajectories of the gradient flow $\dot{x}=-\nabla g(x)$ and  the controlled flows $\dot{x}=u_n(x)$, all initialized at $x(0)=-2$.}
\end{figure}

\subsection{High-dimensional benchmark problems}
We further implement the algorithm for high-dimensional Rastrigin and Ackley benchmarks. We consider $N=50$ particles and initial distribution $\rho_0=\mathcal{U}[-1,-0.5]^d$. The expectation is computed with $100$ realizations of the controlled-CBO algorithm, see results in \Cref{table:hd_Ras} for Rastrigin functions and \Cref{table:hd_Ackley} for Ackley functions. Due to limitations in the computational efficiency of the HJB solver, controlled CBO method is not expected to handle extremely high-dimensional problems. However, satisfactory numerical results have been obtained for 8-dimensional problems with truncated total degree 6, and for tests up to 30 dimensions with only 50 particles, by using the hyperbolic cross basis to mitigate the curse of dimensionality.
\begin{table}[h]
\begin{subtable}[c]{\textwidth}
\centering
\scalebox{0.75}{
\begin{tabular}{|l|l|l|l|l|l|l|}
\hline
$\mathbb{E}[W_2^2(\rho_t^N,\delta_{x^*})]$& $d=2$      & $d=4$      & $d=6$      & $d=8$      & $d=10$     & $d=30$     \\ \hline
$J=2$             & $8.43\times 10^{-28}$ & $7.99\times 10^{-29}$ & $1.15\times 10^{-30}$ & $3.63\times 10^{-31}$ & $1.79\times 10^{-29}$ & $5.06\times 10^{-30}$ \\ \hline
$J=4$                    & $1.81\times 10^{-28}$ & $4.14\times 10^{-28}$ & $1.41\times 10^{-32}$ & $3.52\times 10^{-29}$ & $1.75\times 10^{-27}$ & $1.65\times 10^{-29}$ \\ \hline
\end{tabular}
}
\subcaption{Hyperbolic cross basis generated by Legendre polynomials with the maximum degree $J$}
\end{subtable}
\begin{subtable}[c]{\textwidth}
\centering
\scalebox{0.75}{
\begin{tabular}{|l|l|l|l|l|l|l|}
\hline
$\mathbb{E}[W_2^2(\rho_t^N,\delta_{x^*})]$& $d=2$      & $d=4$      & $d=6$      & $d=8$      & $d=10$     & $d=30$     \\ \hline
$J=2$       &$1.42\times 10^{-18}$      & $2.72\times 10^{-18}$ & $4.06\times 10^{-18}$ & $5.50\times 10^{-18}$ & $7.42\times 10^{-18}$ & $1.74\times 10^{-17}$  \\ \hline
$J=4$     &$1.35\times 10^{-19}$        & $1.49\times 10^{-19}$ & $2.54\times 10^{-19}$ & $4.11\times 10^{-19}$ & $6.02\times 10^{-19}$ & $2.65\times 10^{-18}$ \\ \hline
\end{tabular}
}
\subcaption{Hyperbolic cross basis generated by monomials with the maximum degree $J$.}
\end{subtable}
\begin{subtable}[c]{\textwidth}
\centering
\scalebox{0.75}{
\begin{tabular}{|l|l|l|l|l|}
\hline
    & $d=2$      & $d=4$      & $d=6$      & $d=8$     \\ \hline
$M=2$ & $2.72\times 10^{-28}$ & $2.14\times 10^{-31}$ & $3.11\times 10^{-31}$ & $3.63\times 10^{-31}$ \\ \hline
$M=4$ & $1.17\times 10^{-29}$ & $3.05\times 10^{-29}$ & $1.90\times 10^{-28}$ & $1.53\times 10^{-27}$ \\ \hline
$M=6$ & $9.70\times 10^{-27}$ & $2.53\times 10^{-28}$ & $4.13\times 10^{-28}$ & $2.67\times 10^{-27}$ \\ \hline
\end{tabular}
}
\subcaption{Full multidimensional basis generated by Legendre polynomials and truncated by total degree $M$.}
\end{subtable}
%\vspace{-0.4cm}
\caption{Numerical results for the controlled-CBO in $d$-dimensional Rastrigin function with different basis.}
\label{table:hd_Ras}
\end{table}

\begin{table}[h]
\centering
\scalebox{0.75}{
\begin{tabular}{|l|l|l|l|l|l|l|}
\hline
$\mathbb{E}[W_2^2(\rho_t^N,\delta_{x^*})]$ & $d=2$                & $d=4$                & $d=6$                & $d=8$                & $d=10$ & $d=30$ \\ \hline
$J=2$                                      & $6.30\times 10^{-7}$ & $5.31\times 10^{-6}$ & $8.38\times 10^{-6}$ & $1.45\times 10^{-5}$ & $ 3.03\times 10^{-5}$  &   $7.27\times 10^{-4}$     \\ \hline
$J=4$                                      & $7.17\times 10^{-7}$ & $6.77\times 10^{-6}$ & $2.80\times 10^{-6}$ & $1.23\times 10^{-5}$ &   $1.04\times 10^{-5}$     &     $9.86\times 10^{-4}$    \\ \hline
\end{tabular}
}
\caption{Numerical results for the controlled-CBO in $d$-dimensional Ackley function with $N_m=10^6$ (number of uniform samples for Monte Carlo integration). The HJB approximation uses hyperbolic cross basis generated by monomials with the maximum degree $J$.}
\label{table:hd_Ackley}
\vspace{-0.1cm}
\end{table}

\section{Conclusions}
\label{sec:conclusions}
We  developed a consensus-based global optimization method incorporating an accelerating feedback control term. This proposed controlled-CBO method exhibits faster convergence and remarkably higher accuracy than existing CBO algorithms
for standard benchmark functions. In particular, the controlled-CBO method is well suited if the particle system is poorly initialized and when using only a small number of particles. Future work will focus on
establishing theoretical results related to the convergence rate of the controlled-CBO.  
This approach has potential applications in complex engineering problems, such as shape optimization, optimal actuator/sensor placement \cite{kalise2024multi}, and hyper-parameter tuning in machine learning. In future work, we will study the practical impact of the proposed methodology on the aforementioned problems, 
where classical gradient-based methods typically fail due to non-convex energy landscapes and the presence of  numerous local minimizers. Most notably, the optimal feedback control can be introduced into various classical optimization methods to enhance performance and overcome non-convexity.  However, the scalability of control-based methods currently limited by the computational efficiency of the HJB solver, making it less suitable for extremely high-dimensional problems. Future work may also focus on enhancing the efficiency of the HJB solvers
to extend the applicability of control-based methods to larger-scale systems. 

\appendix

\section{Convergence of \cref{alg1}}\label{sec:appx_vi}
 Before proving Proposition \ref{main}, we first show that given arbitrary control $u\in\mathcal{A}(\Omega)$,  there exists a {unique} solution to the GHJB equation $\mathcal{G}_\mu(V,DV,u)=0$.

\begin{lemma}
\label{lemma_existence}
Assume that \Cref{ass_convergence}\ref{1} holds, %$u\in\mathcal{A}(\Omega)$.
then for any admissible control $u\in\mathcal{A}(\Omega)$, there exists a unique  continuously differentiable solution $V$ to the equation $\mathcal{G}_\mu(V,DV,u)=0$.
\end{lemma}
\begin{proof}
For any admissible control $u\in\mathcal{A}(\Omega)$, given $x\in\Omega$, \Cref{ass_convergence}\ref{1} guarantees that the solution to the HJB equation (\ref{HJB}) and the trajectory $y(t)\equiv y(t;x,u)$ of the system (\ref{control}) are well-defined \cite[Chapter 8.8]{bardi1997optimal}. Now, define $V(x)= \int_0^{\infty}e^{-\mu t} 
\left(f\left(y(t)\right)+\frac{\epsilon}{2}|u(t)|^2\right) d t$,  for any $t\geq 0,$ we have 
$$V\left(x\right) 
=\int_0^t e^{-\mu s} \left(f(y(s))+ \frac{\epsilon}{2}|u(s)|^2\right) d s+e^{-\mu t} V\left(y(t)\right).
$$
Since the map $t \mapsto V\left(y(t)\right)$ is absolutely continuous, rearranging above for $V(y(t))$ and differentiating yields
\begin{equation}
\frac{\partial}{\partial t}V\left(y(t)\right)  
=\mu V\left(y(t)\right)-\left(f(y(t))+ \frac{\epsilon}{2}|u(t)|^2\right).
\label{V_dot}\end{equation}
Combining with the fact $\frac{\partial}{\partial t}V=\frac{\partial V^\top}{\partial y} \dot{y}=D V u $, we obtain that $V$ satisfies the GHJB equation $\mathcal{G}_{\mu}(V,DV;u)=-\mu V+DV u+f+\frac{\epsilon}{2}|u|^2=0$. Since  $u$ is continuous by the admissibility assumption, $f$ is continuous by \Cref{ass_convergence}\ref{1}, and $V$ is continuous by definition,  then $V$ is continuously differentiable. Uniqueness follows using a standard contradiction argument as the HJB equation is assumed to possess a unique solution.
\end{proof}

\begin{figure}  \centering\includegraphics[scale=0.4]{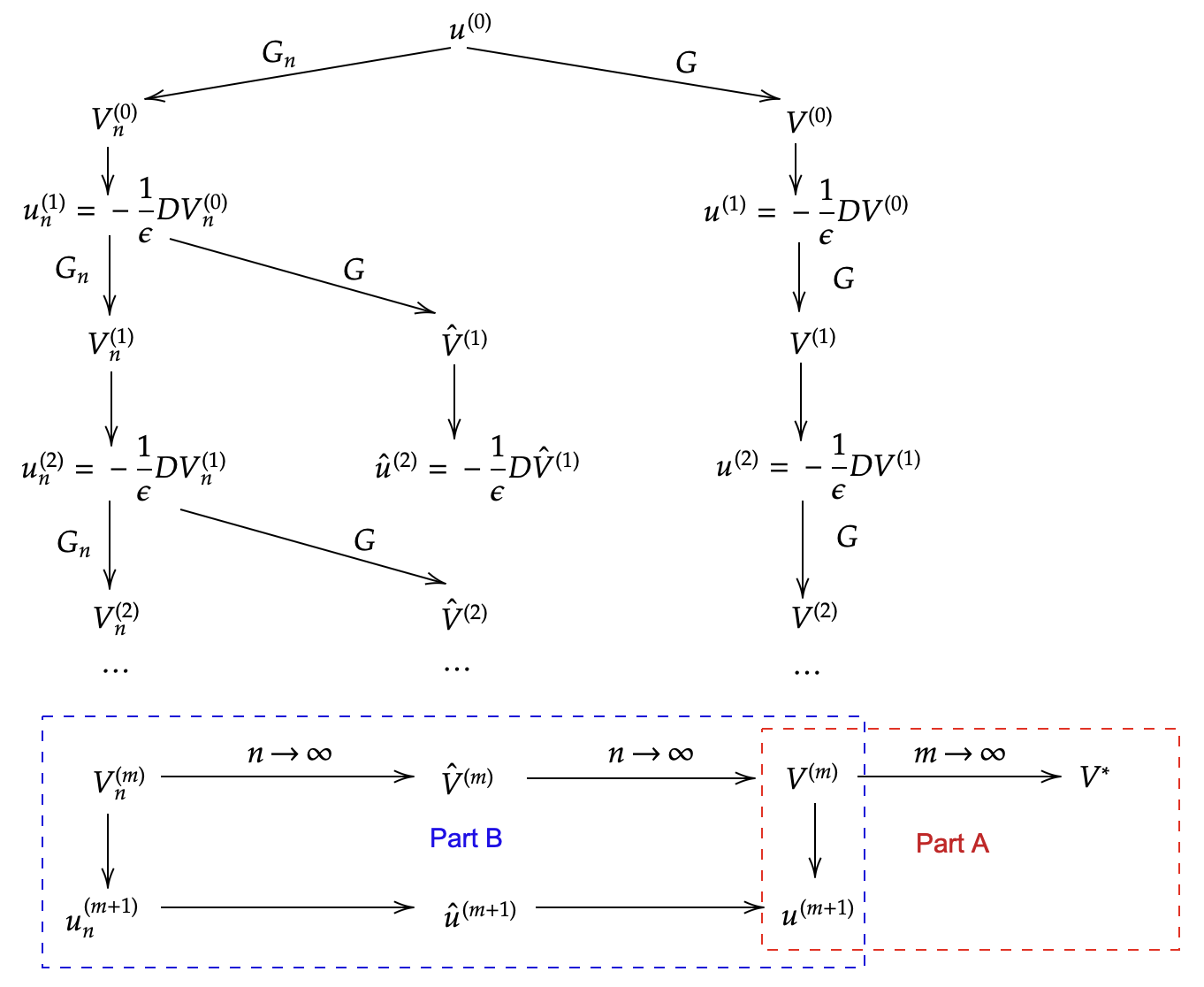}
  \caption{Convergence sketch of \Cref{alg1}: $G_n:\mathcal{A}(\Omega)\rightarrow \operatorname{span}\{\phi_i\}_{i=1}^n$ is an operator that maps any admissible control $u$ to $V_n =\sum_{i=1}^n c_i \phi_i$ satisfying $\left\langle\mathcal{G}_\mu\left(V_n, DV_n; u\right), \Phi_n\right\rangle=0$. Similarly $G$ maps admissible control onto solutions of the GHJB equation, i.e., $V=G u$ implies that $\mathcal{G}_\mu\left(V, DV; u\right)=0$.}
  \label{ill_convergence}
\end{figure}
To prove Proposition \ref{main}, note that 
\begin{displaymath}
|V_n^{(m)}-V^*|_{L^2(\Omega)} \leq|V_n^{(m)}-V^{(m)}|_{L^2(\Omega)}+|V^{(m)}-V^{*}|_{L^2(\Omega)},
\end{displaymath}
where $V^{(m)}=\sum_{i=1}^\infty c_i^{(m)}\phi_i$ is the  solution of the GHJB equation  given $u^{(m)}$, i.e., $\mathcal{G}_\mu\left(V^{(m)}, DV^{(m)};u^{(m)}\right)=0$. The convergence of the \Cref{alg1} consists of two parts (A and B), as  illustrated in \Cref{ill_convergence}. 

\textbf{Part A} covers $V^{(m)}\rightarrow V^*$, as the number of iterations $m\rightarrow\infty$. We will follow the steps in \cite[Proposition 1]{kundu2024policy} to prove  $V^*\leq V^{(m+1)}\leq V^{(m)}$, which implies that according to \Cref{alg1}, the updated control improves the performance of the system at every iteration. 

\textbf{Part B} is the convergence of the Galerkin approximation $V_n^{(m)}\rightarrow V^{(m)}$,  as the degree of approximation $n\rightarrow\infty$. 
 Note that 
\begin{equation}
\left|V_n^{(m)}-V^{(m)}\right|_{L_2(\Omega)} \leq\left|V_n^{(m)}-\hat{V}^{(m)}\right|_{L_2(\Omega)}+\left|\hat{V}^{(m)}-V^{(m)}\right|_{L_2(\Omega)}, 
\label{eq:convergence_V}
\end{equation}
where $V_n^{(m)}=\sum_{i=1}^n c_i^{(m)}\phi_i$ is the solution to the projected GHJB equation with the approximate control $u_n^{(m)}$, i.e.,
$\left\langle\mathcal{G}_\mu\left(V_n^{(m)}, DV_n^{(m)}; u_n^{(m)}\right), \Phi_n\right\rangle=0,$  and $\hat{V}^{(m)}=\sum_{i=1}^\infty \hat{c}_i^{(m)}\phi_i$ is the solution to the GHJB equation $\mathcal{G}_\mu\left(\hat{V}^{(m)}, D\hat{V}^{(m)};u_n^{(m)}\right)=0.$
Note that  $\hat{V}^{(m)}$ depends on $n$ through $u_n^{(m)}$ and it is an ancillary variable for the proof,
 not be used in practice. 

\subsection{Part A}
\label{sec:part A}
We present a similar result to \cite[Proposition 1]{kundu2024policy}.
\begin{proposition}
Assume \Cref{ass_convergence}\ref{1} and \ref{2} hold, given $u^{(0)} \in \mathcal{A}(\Omega)$, then $u^{(m+1)}=-\frac{1}{\epsilon}D V^{(m)}\in \mathcal{A}(\Omega)$ for all $m$, and we have $V^{(m+1)}(x) \leq V^{(m)}(x)$ for any $x\in\Omega$.
\label{partA}
\end{proposition}
\begin{proof}
We proceed by induction method. Given 
$u^{(0)} \in \mathcal{A}(\Omega)$, we assume $u^{(m)}$ is admissible, then by Lemma \ref{lemma_existence}, we know that there exists a unique solution $V^{(m)}$ to equation $\mathcal{G}_\mu(V^{(m)},DV^{(m)},u^{(m)})=0$ and $V^{(m)}$ is continuously differentiable,  so $u^{(m+1)}=-\frac{1}{\epsilon}DV^{(m)}$ is continuous. Due to   \Cref{ass_convergence}\ref{1}, the function $f$ and $u^{(m+1)}$ are  bounded on $\Omega$, then
$\int_0^\infty e^{-\mu t} (f(y(t))+|u^{(m+1)}(t)|^2)dt$ is bounded above.
Therefore $u^{(m+1)}$ is admissible. Then, for every $x\in \Omega$, we denote the trajectory of system (\ref{control}) as $y^{(m+1)}_t\equiv y(t;x,u^{(m+1)})$.
Since $V^{(m)}$ is a solution to $\mathcal{G}_\mu(V^{(m)},DV^{(m)};u^{(m)})$, and $u(\cdot)$ is in feedback form, we obtain 
$$
\begin{aligned}V^{(m)}(x)&=\int_0^{\infty}e^{-\mu t} 
\left(f\left(y^{(m+1)}_t\right)+\frac{\epsilon}{2}|u^{(m)}(y^{(m+1)}_t)|^2\right) d t\\
&= \int_0^{\infty}e^{-\mu t} 
\left(\mu V^{(m)}(y^{(m+1)}_t)-D V^{(m)}(y^{(m+1)}_t) u^{(m+1)}\right) d t,\end{aligned}$$
Similarly for the solution $V^{(m+1)}$ of $\mathcal{G}_\mu(V^{(m+1)},DV^{(m+1)};u^{(m+1)})$:
$$V^{(m+1)}(x)=\int_0^{\infty}e^{-\mu t} 
\left(\mu V^{(m+1)}(y^{(m+1)}_t)-D V^{(m+1)}(y^{(m+1)}_t) u^{(m+1)}\right)  d t.
$$
Rearranging the terms in $\mathcal{G}_\mu(V^{(m)},DV^{(m)};u^{(m)})$ and $\mathcal{G}_\mu(V^{(m+1)},DV^{(m+1)};u^{(m+1)})$, and substituting  into the difference of $V^{(m+1)}(x)$ and $V^{(m)}(x)$, as $u^{(m+1)}=-\frac{1}{\epsilon}D V^{(m)}$, we have $V^{(m+1)}(x)-V^{(m)}(x)$
\begin{displaymath}
\begin{aligned}
&=\int_0^{\infty}e^{-\mu t}\left(\frac{\epsilon}{2}|u^{(m+1)}|^2-\frac{\epsilon}{2}|u^{(m)}|^2-\epsilon u^{(m+1)}\left(u^{(m+1)}-u^{(m)}\right)\right) d t\\
&=-\int_0^{\infty}e^{-\mu t}\left(\frac{\epsilon}{2}|u^{(m+1)}-u^{(m)}|^2\right) d t\leq 0.
\end{aligned}
\end{displaymath}
\end{proof}
\Cref{partA} and \Cref{lemma_existence} imply that  for any $m$, given $u^{(m)}$, there exists  a $V^{(m)}=\sum_{i=1}^{\infty} c_i^{(m)} \phi_i$ such that $\mathcal{G}_\mu\left(V^{(m)}, DV^{(m)}; u^{(m)}\right)=0$, and
$V^{(m)}$ converges to $V^*$ uniformly. Additionally, $u^{(m+1)}=-\frac{1}{\epsilon}DV^{(m)} \in\mathcal{A}(\Omega)$.

\subsection{Part B}\label{sec:conv_part_B}

For any  admissible control $u\in\mathcal{A}(\Omega)$, we denote the actual solution $\hat{V}=\sum_{i=1}^\infty \hat{c}_i\phi_i$ satisfying $
\mathcal{G}_\mu(\hat{V},D\hat{V}; u)=0
$, and the approximation $V_n=\sum_{i=1}^n c_i\phi_i$ satisfying $
\langle\mathcal{G}_\mu(V_n,DV_n; u),\Phi_n\rangle=0
$. The latter is  a weaker condition than $\mathcal{G}_\mu\left(V_n, DV_n ; u\right)=0$,  ensuring the error of approximation, projected onto $\operatorname{span}\left\{\phi_i\right\}_{i=1}^n$, is zero.
As mentioned earlier, the convergence of Galerkin approximation has been thoroughly studied  for the un-discounted infinite horizon case in \cite{beard1995improving,beard1997galerkin, beard1998approximate} with GHJB equation $\mathcal{G}_0(V, D V ; u):=D V^\top u+f+\frac{\epsilon}{2}|u|^2.$ Therefore 
the results in this part can be proved by following the proof techniques in \cite{beard1995improving,beard1997galerkin, beard1998approximate}, with adjustments made for the different form of GHJB equation due to the presence of $\mu V$. For the convenience of the reader, we present the results and detailed proofs. In the following proofs, we assume $\{\phi_i\}_{i=1}^\infty$ is orthonormal as in \cite[Lemma 15, 16]{beard1997galerkin} and \cite[Lemma 5.2.9 ]{beard1995improving}.
\begin{lemma}
\label{convergence_GHJB}
Given $u \in \mathcal{A}(\Omega)$, under the Assumption \ref{1}-\ref{4}, \ref{c_bdd} and \ref{6}, we have
$
\left|\mathcal{G}_\mu\left(V_n, DV_n; u\right)\right| \rightarrow 0 
$
uniformly on $\Omega$ as $n\rightarrow\infty$. 
\end{lemma}
\begin{proof}
The proof follows  \cite[Lemma 5.2.13]{beard1995improving}, \cite[Lemma 20]{beard1997galerkin}. First, note that the Assumption \ref{4} implies $\mathcal{G}_\mu\left(V_n,DV_n; u\right) \in \operatorname{span}\left\{\phi_i\right\}_{i=1}^{\infty}$. As we assume $\left\{\phi_i\right\}_{i=1}^{\infty}$ is orthonormal, we have
\begin{displaymath}\begin{aligned}
&\left|\mathcal{G}_\mu\left(V_n,DV_n; u\right)\right|= \left|\sum_{i=1}^{\infty}\left\langle\mathcal{G}_\mu\left(V_n,DV_n; u\right), \phi_i\right\rangle \phi_i\right| \\ 
& \leq\left|\sum_{i=1}^{n}\left\langle\mathcal{G}_\mu\left(V_n,DV_n; u\right), \phi_i\right\rangle \phi_i+\sum_{i=n+1}^{\infty}\left[\left\langle -\mu V_n+DV_n u, \phi_i\right\rangle +\left\langle f+\frac{\epsilon}{2}|u|^2, \phi_i\right\rangle\right]\phi_i \right|\\
& =\left|\sum_{i=n+1}^{\infty}\left[\sum_{k=1}^n c_k\left\langle -\mu \phi_k+\frac{\partial \phi_k}{\partial x}u, \phi_i\right\rangle\phi_i\right] +\sum_{i=n+1}^{\infty}\left\langle f+\frac{\epsilon}{2}|u|^2, \phi_i\right\rangle\phi_i \right|\\
&\leq GH+I,
\end{aligned}\end{displaymath}
where
\begin{displaymath}
\begin{gathered}
G := \sup _{k=1,\cdots,n}\left|c_k\right|,\qquad H:= \sup _{k=1,\cdots,n}\left|\sum_{i=n+1}^{\infty}\left\langle -\mu \phi_k+\frac{\partial \phi_k}{\partial x}u, \phi_i\right\rangle\phi_i\right|,\\
I:=\left|\sum_{i=n+1}^{\infty}\left\langle f+\frac{\epsilon}{2}|u|^2, \phi_i\right\rangle \phi_i\right|.
\end{gathered}
\end{displaymath}
Since $\sum_{i=n+1}^{\infty}\left\langle -\mu \phi_k+\frac{\partial \phi_k^{\top}}{\partial x}u, \phi_i\right\rangle \phi_i(x)$ and $\sum_{i=n+1}^{\infty}\left\langle f+\frac{\epsilon}{2}|u|^2, \phi_i\right\rangle \phi_i(x)$ are continuous and point-wise decreasing by \Cref{ass_convergence}\ref{6}, \cite[Lemma 5.2.12]{beard1995improving} implies  that $H$ and $I$ converge to 0 uniformly as $n\rightarrow\infty$. Combined with \Cref{ass_convergence}\ref{c_bdd}, the term $G$ is uniformly bounded for all $n$, then the desired result follows.
\end{proof}

\begin{lemma}
\label{convergence_c}
Given $u \in \mathcal{A}(\Omega)$, under the assumptions of \Cref{convergence_GHJB} and \Cref{ass_convergence}\ref{new}, we have $
\left|\boldsymbol{c}_n-\hat{\boldsymbol{c}}_n\right| \rightarrow 0$ as $n\rightarrow\infty$, where $\mathbf{c}_n=\{c_i\}_{i=1}^{n}, \;\hat{\mathbf{c}}_n=\{\hat{c}_i\}_{i=1}^{n}$.
\end{lemma}
\begin{proof}
The proof follows  \cite[Lemma 22]{beard1997galerkin} and \cite[Lemma 5.2.15]{beard1995improving}. 
Since for any $u\in\mathcal{A}(\Omega)$, 
$\mathcal{G}_\mu(\hat{V},D\hat{V};u)=0$, we have
\begin{displaymath}
\begin{aligned}
\mathcal{G}_\mu\left(V_n,DV_n; u\right)-\mathcal{G}_\mu(\hat{V},D\hat{V}; u)&=\mathcal{G}_\mu\left(V_n,DV_n; u\right) \\
\left(\boldsymbol{c}_n -\hat{\boldsymbol{c}}_n\right)^\top(-\mu\Phi_n+
\nabla  \Phi_n u)&=\mathcal{G}_\mu\left(V_n,DV_n; u\right)+\sum_{i=n+1}^{\infty} \hat{c}_i (-\mu \phi_i+\frac{\partial \phi_i}{\partial x}u)
\end{aligned}
\end{displaymath}
By the mean value theorem, there exists  $\xi \in \Omega$ such that
\begin{displaymath}
\begin{aligned}
&\left|\left(\boldsymbol{c}_n -\hat{\boldsymbol{c}}_n\right)^\top(-\mu\Phi_n+
\nabla  \Phi_n u)\right|^2_{L_2(\Omega)}\\
=&\int_{\Omega}\left|\mathcal{G}_\mu\left(V_n,DV_n; u\right)(x)+\sum_{i=n+1}^{\infty} \hat{c}_i \left(-\mu \phi_i+\frac{\partial \phi_i}{\partial x} u\right)(x)\right|^2 d x \\
\leq & |\Omega|\left(2\left|\mathcal{G}_\mu\left(V_n,DV_n; u\right)(\xi)\right|^2+2\left|\sum_{i=n+1}^{\infty} \hat{c}_i \left(-\mu \phi_i+\frac{\partial \phi_i}{\partial x} u\right)(\xi)\right|^2\right),
\end{aligned}
\end{displaymath}
where $|\Omega|$ is the Lebesgue measure of $\Omega$. Lemma \ref{convergence_GHJB} implies that for $\forall \delta>0, \exists K_1$ such that
as $n>K_1$ we have $|\mathcal{G}_\mu\left(V_n,DV_n; u\right)(x)|<\frac{\delta}{\sqrt{|\Omega|}}.$
Since $\sum_{i=1}^{\infty} \hat{c}_i\left(-\mu \phi_i+ \frac{\partial \phi_i^\top}{\partial x} u\right)=-f-\frac{\epsilon}{2}|u|^2
$, then for $\forall \delta>0,$ under \Cref{ass_convergence}\ref{5}, $\exists K_2$ such that as $n>K_2$, we have 
$\left|\sum_{i=n+1}^{\infty} \hat{c}_i \left(-\mu \phi_i+\frac{\partial \phi_i^\top}{\partial x} u\right)(x)\right|<\frac{\delta}{\sqrt{|\Omega|}}.$
Thus, as $n\rightarrow\infty$, we obtain
\begin{equation}
\label{cphi_0}
\left|\left(\boldsymbol{c}_n -\hat{\boldsymbol{c}}_n\right)^\top(-\mu\Phi_n+
\nabla\Phi_n u)\right|^2_{L^2(\Omega)}\rightarrow0 \text{  uniformly on  }\Omega.
\end{equation}
By \Cref{ass_convergence}\ref{new}, we know $-\mu \phi_i+\frac{\partial \phi_i}{\partial x} u\not\equiv 0$  and $\{-\mu \phi_i+\frac{\partial \phi_i}{\partial x} u\}_{i=1}^\infty$ is linearly independent. Thus \Cref{cphi_0} implies the desired results $\left|\boldsymbol{c}_n-\hat{\boldsymbol{c}}_n\right| \rightarrow 0$.
\end{proof}

\begin{corollary}
\label{convergence_V}
Under the assumptions of  \Cref{convergence_c}, then
$ |V_n-\hat{V}|_{L_2(\Omega)} \rightarrow 0.  $
\end{corollary}
\begin{proof}
The proof follows  \cite[Corollary 5.2.16]{beard1995improving}, \cite[Corollary 23]{beard1997galerkin}. 
Note that $V_n=\sum_{i=1}^n c_i\phi_i$ and $\hat{V}=\sum_{i=1}^\infty \hat{c}_i\phi_i$,  we have 
\begin{displaymath}
\begin{aligned}
\left|V_n-\hat{V}\right|_{L_2(\Omega)}^2
& \leq 2\left|\sum_{i=1}^n\left(c_i-\hat{c}_i\right)\phi_i\right|_{L_2(\Omega)}^2+2\left|\sum_{i=n+1}^{\infty} \hat{c}_i \phi_i\right|_{L_2(\Omega)}^2 \\
& =
2\left(\boldsymbol{c}_n -\hat{\boldsymbol{c}}_n
\right)^\top\left\langle\Phi_n, \Phi_n^\top\right\rangle\left(
\boldsymbol{c}_n -\hat{\boldsymbol{c}}_n
\right)+2\int_{\Omega} \left|\sum_{i=n+1}^{\infty} \hat{c}_i \phi_i(x)\right|^2 d x
\end{aligned}
\end{displaymath}
Combined with \Cref{convergence_c},  the mean value theorem implies there exists $ \xi \in \Omega$ such that as $n\rightarrow\infty,$
\begin{displaymath}
\left|V_n-\hat{V}\right|_{L_2(\Omega)}^2  =2\left|\boldsymbol{c}_n-\hat{\boldsymbol{c}}_n\right|^2+2|\Omega|\left|\sum_{i=n+1}^{\infty} \hat{c}_i \phi_i(\xi)\right|^2  \rightarrow 0
\end{displaymath}
\end{proof}

\begin{corollary}
\label{convergence_u}
Under the Assumptions of \Cref{convergence_GHJB} and \Cref{ass_convergence}\ref{5},  $\left|u_n-\hat{u}\right| \rightarrow 0$ uniformly on $\Omega$, where $u_n:=-\frac{1}{\epsilon}DV_n$ and $\hat{u}:=-\frac{1}{\epsilon}D\hat{V}$. 
\end{corollary}
\begin{proof}The proof follows \cite[Lemma 5.2.17]{beard1995improving}, \cite[Lemma 24]{beard1997galerkin}. By definition of $u_n:=-\frac{1}{\epsilon}DV_n$ and $\hat{u}:=-\frac{1}{\epsilon}D\hat{V}$, we have
\begin{displaymath}
\left|u_n-\hat{u}\right|\leq\left|-\frac{1}{\epsilon}  \sum_{i=1}^n\left(c_i-\hat{c}_i\right)\frac{\partial \phi_i}{\partial x}\right|+\left|\frac{1}{\epsilon} \sum_{i=n+1}^{\infty} \hat{c}_i \frac{\partial \phi_i}{\partial x}\right|.
\end{displaymath}
The second term on the right hand side converges pointwise to 0 and uniformly if \Cref{ass_convergence} \ref{5} is satisfied.
Then the uniform convergence of $|\mathbf{c}_n-\hat{\mathbf{c}}|\rightarrow 0$ implies the uniform convergence of $|\nabla\Phi_n^\top(\mathbf{c}_n-\hat{\mathbf{c}})|\rightarrow 0$.
\end{proof}

\begin{lemma}\label{lem:u_admis}
Under the Assumptions of  \Cref{convergence_u}, for $n$ sufficiently large, we have $u_n\in\mathcal{A}(\Omega).$
\end{lemma}
\begin{proof}
First, note that since $u_n(x)  =-\frac{1}{\epsilon}\frac{\partial V_n}{\partial x}(x)$ and $V_n$ is continuously differentiable, $u_n$ is continuous. Then since $u_n\rightarrow \hat{u}$ uniformly as $n\rightarrow \infty,$ which means that for any $\delta>0$, there exists $N>0$ s.t when $n>N$, $|u_n-\hat{u}|<\delta$ and $|u_n|<|\hat{u}|+\delta$. Therefore the admissibility of $u_n$ will follow from the admissibility of $\hat{u}$.  From \Cref{partA}, when $u^{(m)}=u$, we have $u^{(m+1)}=\hat{u}$ is admissible. 
\end{proof}

\begin{theorem}\label{thm:unif_conv_V_u}
  Under Assumption \ref{ass_convergence}, for each integer $m \geq 0$ and as $n\rightarrow\infty$, we have
\begin{displaymath}
\begin{aligned}
& \left|V_n^{(m)}-V^{(m)}\right|_{L^2(\Omega)} \rightarrow 0, \qquad \sup _{x \in \Omega}\left|u_n^{(m+1)}(x)-u^{(m+1)}(x)\right| \rightarrow 0,
\end{aligned}
\end{displaymath}
and moreover,  $u_n^{(m+1)} \in \mathcal{A}(\Omega)$ for all sufficiently large $n$.
\end{theorem}
\begin{proof}[Proof of \Cref{thm:unif_conv_V_u}]
We will use the induction method similar to \cite[Theorem 4.2]{beard1998approximate}, \cite[Theorem 5.3.6]{beard1995improving}.
\\
\textbf{Initial Step.} It has been proved in \Cref{sec:part A} that for $u^{(0)}\in\mathcal{A}(\Omega)$, there exists  a $V^{(0)}=\hat{V}^{(0)}$ such that $\mathcal{G}_\mu\left(V^{(0)}, DV^{(0)}; u^{(0)}\right)=0$. Then,  by \Cref{convergence_V} and \Cref{convergence_u}, we have 
\begin{displaymath}
\begin{aligned}
    &\left|V_n^{(0)}-V^{(0)}\right|_{L^2(\Omega)}=\left|V_n^{(0)}-\hat{V}^{(0)}\right|_{L^2(\Omega)} \rightarrow 0\\
    &\sup _{x\in\Omega}\left|u_n^{(1)}(x)-u^{(1)}(x)\right|=\sup _{x\in\Omega}\left|u_n^{(1)}(x)-\hat{u}^{(1)}(x)\right| \rightarrow 0\\
    & u_n^{(1)}\in\mathcal{A}(\Omega)
\end{aligned}
\end{displaymath}
\textbf{Induction Step.} Assume that
\begin{displaymath}
\begin{aligned}
& \left|V_n^{(m-1)}-V^{(m-1)}\right|_{L^2(\Omega)} \rightarrow 0, \qquad \sup _{x\in\Omega}\left|u_n^{(m)}(x)-u^{(m)}(x)\right| \rightarrow 0 \text{ as }n\rightarrow\infty,\\
& u_n^{(m)} \in \mathcal{A}(\Omega), \quad \text { for } n \text { sufficiently large. }
\end{aligned}
\end{displaymath}
First, by taking $u=u_n^{(m)}$, $V_n=V_n^{(m)}$, $\hat{V}=\hat{V}^{(m)}$ in \Cref{convergence_V} and \Cref{convergence_u}, it can be seen that 
\begin{displaymath}
\begin{aligned}
& \left|V_n^{(m)}-\hat{V}^{(m)}\right|_{L^2(\Omega)} \rightarrow 0, \qquad \sup _{\Omega}\left|u_n^{(m+1)}-\hat{u}^{(m+1)}\right| \rightarrow 0\text{ as }n\rightarrow\infty, \\
& u_n^{(m+1)} \in \mathcal{A}(\Omega), \quad \text { for } n \text { sufficiently large. }
\end{aligned}
\end{displaymath}
Recall \Cref{eq:convergence_V}, 
it remains to show that
\begin{displaymath}\left|V^{(m)}-\hat{V}^{(m)}\right|_{L^2(\Omega)} \rightarrow 0\quad\text{and} \quad \sup _{\Omega}\left|u^{(m+1)}-\hat{u}^{(m+1)}\right| \rightarrow 0 \text{ as }n\rightarrow\infty.
\end{displaymath}
From the induction step, we have
$
\sup _{\Omega}\left|u_n^{(m)}-u^{(m)}\right| \rightarrow 0
$ uniformly on $\Omega$. Since the trajectory $y(t;x,u)$ depends continuously on control $u$, this implies that for $n$ sufficiently large and $\forall x\in\Omega$, 
$
V^{(m)}=\int_0^{\infty} e^{-\mu t}\left[f\left(y\left(t;x,u_n^{(m)}\right)\right)+
\frac{\epsilon}{2}\left| u_n^{(m)}(t)\right|^2\right] d t
$
is uniformly close to
$
\hat{V}^{(m)}=\int_0^{\infty}e^{-\mu t} \left[f\left(y\left(t;x,u^{(m)}\right)\right)+\frac{\epsilon}{2}\left| u^{(m)}(t)\right|^2\right] d t .
$
Similar to \Cref{eq:convergence_V}, we also have 
\begin{displaymath}
\sup_\Omega\left|u_n^{(m+1)}-u^{(m+1)}\right| \leq \sup_\Omega\left|u_n^{(m+1)}-\hat{u}^{(m+1)}\right| +  \sup_\Omega\left|\hat{u}^{(m+1)}-u^{(m+1)}\right|.
\end{displaymath}
Since \Cref{convergence_u} implies that $\sup_\Omega\left|u_n^{(m+1)}-\hat{u}^{(m+1)}\right| \rightarrow 0$ as $n\rightarrow\infty$, so the proof reduces to show
$
\sup_\Omega \left|u^{(m+1)}-\hat{u}^{(m+1)}\right|\rightarrow 0,$ given $\sup_\Omega \left|u^{(m)}-u_n^{(m)}\right| \rightarrow 0.$  By subtracting $\mathcal{G}_\mu(V^{(m)},DV^{(m)};u^{(m)})$ and $\mathcal{G}_\mu(\hat{V}^{(m)},D\hat{V}^{(m)};u_n^{(m)})$, and subtracting $D\hat{V}^{(m)}u^{(m)}$ on both sides of the equation we obtain
\begin{displaymath}
|-\mu(V^{(m)}-\hat{V}^{(m)})+(DV^{(m)}-D \hat{V}^{(m)}) u^{(m)}| \leq|D\hat{V}^{(m)}| |u_n^{(m)}-u^{(m)}| +\frac{\epsilon}{2}|u_n^{(m)}-u^{(m)}|^2 .
\end{displaymath}
Since $D \hat{V}^{(m)}$ is continuous on $\Omega$, $\left|D\hat{V}^{(m)}\right|$ is uniformly bounded on the compact set $\Omega$. By induction hypothesis $\sup _{\Omega}\left|u_n^{(m)}-u^{(m)}\right| \rightarrow 0$, the right-hand side of the above inequality goes to 0  as $n\rightarrow \infty$. Thus,
\begin{displaymath}
\sup_\Omega\left|-\mu(V^{(m)}-\hat{V}^{(m)})+(DV^{(m)}-D\hat{V}^{(m)}) u^{(m)}\right|\rightarrow 0,
\end{displaymath}
which  is equivalent to
\begin{equation}
\label{eq:c}
\sup_\Omega\left|\sum_{i=1}^{\infty}\left[c_i^{(m)}-\hat{c}_i^{(m)}\right]\left(-\mu \phi_i+\frac{\partial \phi_i}{\partial x}  u^{(m)} \right)\right| \rightarrow 0 .
\end{equation}
The \Cref{ass_convergence}\ref{new}  implies that $-\mu \phi_i+\frac{\partial \phi_i}{\partial x} u^{(m)}\not\equiv 0$ (note that this term does not depend on $n$), and $\{-\mu \phi_i+\frac{\partial \phi_i}{\partial x} u^{(m)}\}_{i=1}^\infty$ is linearly independent. Therefore, equation (\ref{eq:c}) is equivalent to $|c_i^{(m)}-\hat{c}_i^{(m)}|\rightarrow 0$ as $n\rightarrow\infty.$
Finally, we have \begin{displaymath}
\sup \left|u^{(m+1)}-\hat{u}^{(m+1)}\right|=\sup \left|-\frac{1}{\epsilon}\sum_{i=1}^\infty\left(c_i^{(m)}-\hat{c}_i^{(m)}\right)\frac{\partial \phi_i}{\partial x}\right|\rightarrow 0,\end{displaymath}
which completes the proof.

  \end{proof}

\begin{remark}\emph{Discussion on the assumptions:} for the convergence of Galerkin approximation $V_n^{(m)}\rightarrow V^{(m)}.$ Assumption \ref{2} ensures that the iterative process is able to commence.
Assumption \ref{3}, \ref{4} and \ref{5} 
guarantee that $V$ and all components of the GHJB equation can be approximated arbitrarily close by linear combinations of $\phi_i$.  Assumption \ref{c_bdd} is same to \cite[Assumption A5.5]{beard1995improving} and can be further proved following the steps in \cite[Lemma 20]{beard1997galerkin}. Assumption \ref{new} means that $-\mu\Phi_n+\nabla\Phi_n\cdot u\not\equiv0$ for arbitrary $u\in\mathcal{A}(\Omega)$, this guarantees the linear independence of set $\left\{-\mu\phi_i+\partial \phi_i/ \partial x\cdot u\right\}_{i=1}^N$ and avoids the extreme case that the objective function $f\equiv 0.$ There are similar assumptions in \cite[Corollary 12]{beard1997galerkin} and \cite[Corollary 5.2.7]{beard1995improving}, both of which  utilise the linear independence of set $\left\{\partial \phi_i/ \partial x\cdot u\right\}_{i=1}^N$ in proving the convergence. Assumption
\ref{6} implies that the tail of the infinite series decreases in some uniform manner. This assumption is necessary and sufficient conditions for pointwise convergence to imply uniform convergence on a compact set (see \cite[Lemma 5.2.12]{beard1995improving}, \cite[Lemma 19]{beard1997galerkin}). Note that \cite{beard1995improving,beard1997galerkin,beard1998approximate} further assume 
$\left\{\left| \frac{\partial^2 \phi_i}{\partial x^2}(0)\right|\right\}_{i=1}^{\infty}$ is uniformly bounded to prove the admissibility of $u_n$. In our context, the discount factor relaxes the admissibility condition, here we can obtain the admissibility of $u_n$ by the uniform convergence of $u_n\rightarrow\hat{u}$.
\end{remark}

\section{Proof of \cref{thm}} 
\label{app_1}
Before proving the main result, we introduce some standard results from the theory of SDE \cite{arnold1974stochastic}
 to show the existence and uniqueness of solution. 
\begin{theorem}[Theorem 6.22 \cite{arnold1974stochastic}]
Suppose that we have a SDE
\begin{equation}
dX_t=F\left(t, X_t\right) d t+G\left(t, X_t\right)d W_t, \quad X_{0}=x_0,\quad 0\leq t \leq T<\infty,
\label{SDE}
\end{equation}
where $W_t\in\mathbb{R}^m$ is Brownian motion and $X_0$ is a random variable independent of $W_t$, $t \geq 0$. Suppose that the function $F(t, x)\in\mathbb{R}^d$ and the function $G(t, x)\in \mathbb{R}^{d \times m}$ are measurable on $\left[0, T\right] \times$ $\mathbb{R}^d$ and have the following properties: There exists a constant $K>0$ such that
\begin{itemize}
\item Lipschitz condition: for all $t \in\left[0, T\right], x \in \mathbb{R}^d, y \in \mathbb{R}^d$, \newline $|F(t, x)-F(t, y)|+|G(t, x)-G(t, y)| \leq K|x-y|.$
\item Restriction on growth: For all $t \in\left[0, T\right]$ and $x \in \mathbb{R}^d$,
 \newline $|F(t, x)|^2+|G(t, x)|^2 \leq K^2\left(1+|x|^2\right)$
\end{itemize}
Then, equation \cref{SDE} has  a unique solution $X_t\in\mathbb{R}^d$ on $\left[0, T\right]$, continuous with probability 1, that satisfies the initial condition $X_{0}=x_0$; that is, if $X_t$ and $Y_t$ are continuous solutions of \cref{SDE} with the same initial value $x_0$, then
$
\mathbb{P}\left[\sup _{0 \leq t \leq T}| X_t-Y_t|>0\right]=0 .
$
\label{existence}
\end{theorem}
\begin{theorem}[Theorem 7.12 \cite{arnold1974stochastic}]
\label{moment}
Suppose that the assumptions of Theorem \ref{existence} hold and
$
\mathbb{E}|X_0|^{2 n}<\infty,\quad n\in\mathbb{N}.
$ Then, for the solution $X_t$ of the SDE \cref{SDE}, we have
$\mathbb{E}\left|X_t\right|^{2 n} \leq\left(1+\mathbb{E}|X_0|^{2 n}\right) \mathrm{e}^{Ct},$
where $C=2 n(2 n+1) K^2$.
\end{theorem}

 We denote the $p$-Wasserstein distance $W_p$ between two Borel probability measures $\varrho_1, \varrho_2 \in$ $\mathcal{P}_p\left(\mathbb{R}^d\right)$ as
$$
W_p\left(\varrho_1, \varrho_2\right)=\left(\inf _{\pi \in \Pi\left(\varrho_1, \varrho_2\right)} \int\left|x_1-x_2\right|^p d \pi\left(x_1, x_2\right)\right)^{1 / p},
$$
where $\Pi\left(\varrho_1, \varrho_2\right)$ denotes the set of all couplings of $\varrho_1$ and $\varrho_2$, i.e., the collection of all Borel probability measures over $\mathbb{R}^d \times \mathbb{R}^d$ with marginals $\varrho_1$ and $\varrho_2$, respectively. In particular, the $2$-Wasserstein distance between empirical distribution $\rho_t^N:=\frac{1}{N} \sum_{i=1}^N \delta_{X_t^i}$ and $\delta_{x^*}$ is defined as following
$$
W_2^2\left(\rho_t^N, \delta_{x^*}\right)=\int\left|x-x^*\right|^2 d \rho_t^N (x).
$$
This leads to the following result.

\begin{lemma}[Lemma 3.2 \cite{carrillo2018analytical}]\label{3.2}
 Let $f$ satisfy Assumption \ref{assumption} and $\varrho, \hat{\varrho} \in \mathcal{P}_2\left(\mathbb{R}^d\right)$ with
$$
\int|x|^4 d \varrho(x), \quad \int|\hat{x}|^4 d \hat{\varrho}(\hat{x}) \leq R .
$$
Then the following stability estimate holds
$$
\left|v_\alpha(\varrho)-v_\alpha(\hat{\varrho})\right| \leq m_0 W_2(\varrho, \hat{\varrho}),
$$
for a constant $m_0>0$ depending only on $\alpha, L_f$ and $R$.
\label{lemma3}
\end{lemma}

\begin{proof}[Proof of \cref{thm}]
The proof follows the steps in \cite[Theorem~3.1,3.2]{carrillo2018analytical} with modifications made to accommodate the fixed control function $u^* $ in the drift term and the anisotropic diffusion.\\
\textbf{The Fokker-Planck equation in weak sense for fixed $v_t$.}
According to \cref{existence}, for a deterministic function $v_t=v(t) \in \mathcal{C}\left([0, T], \mathbb{R}^d\right)$ and an initial measure $\rho_0 \in \mathcal{P}_4\left(\mathbb{R}^d\right)$, the SDE 
\begin{equation}
\label{sde_bar}
d \bar{X}_t=[-\lambda\left(\bar{X}_t-v_t\right)+\beta u^*(\bar{X}_t)] d t+\sigma \operatorname{Diag}\left(\bar{X}_t-v_t\right) d W_t,\quad \text{law}(\bar{X}_0)=\rho_0,
\end{equation}
has a unique solution if  $F(t,\bar{X}_t):= -\lambda\left(\bar{X}_t-v_t\right)+\beta u^*(\bar{X}_t)$ and $\sigma \operatorname{Diag}(\bar{X}_t-v_t)$ satisfy the  
Lipszhitz condition and restriction on growth. Let $F:=F_1+F_2$, where $F_1:= -\lambda\left(\bar{X}_t-v_t\right)$, $F_2:=\beta u^*(\bar{X}_t)$.
Then, 
$$
\begin{aligned}
|F(x)-F(y)|&\leq|F_1(x)-F_1(y)|+|F_2(x)-F_2(y)|\\
F^2&\leq(F_1+F_2)^2\leq2F_1^2+2F_2^2.
\end{aligned}
$$

It is easily verified that under \cref{assumption}, the SDE \cref{sde_bar} has a unique solution $\bar{X}_t$, continuous with probability 1,  that satisfies the initial condition $\text{law}(\bar{X}_0)=\rho_0$.  
The solution induces  a law $\rho_t=\operatorname{law}\left(\bar{X}_t\right)$. By \cref{moment}  for $n=2$ and regularity assumption of $\rho_0 \in \mathcal{P}_4\left(\mathbb{R}^d\right)$ , we have $ \mathbb{E}\left|\bar{X}_t\right|^4 \leq\left(1+\mathbb{E}\left|\bar{X}_0\right|^4\right) e^{C t}$, where $C$ is constant depend only on $L_f$, $L_u$ and the dimension $d$. In particular, there exists $R<\infty$ such that  $\sup _{t \in[0, T]} \int|x|^4 d \rho_t(x)\leq R$, therefore $\rho_t\in\mathcal{P}_4\left(\mathbb{R}^d\right)$. Combined with the fact $X_t\in\mathcal{C}([0,T],\mathbb{R}^d)$, we have
$\rho \in\mathcal{C}([0,T],\mathcal{P}_4\left(\mathbb{R}^d\right))$. For any $\phi \in \mathcal{C}_b^2\left(\mathbb{R}^d\right)$, by applying Itô's formula to the integral formulation, we derive
\begin{align*}
&\phi(\bar{X}_t)=\phi(\bar{X}_0)+\int_{0}^t\nabla \phi(\bar{X}_s) \cdot\left[-\lambda(\bar{X}_s-v_s)+\beta u^*(\bar{X}_s)\right]d s\\
&+\sigma\int_0^t \nabla \phi\left(\bar{X}_s\right) \cdot \operatorname{Diag}\left(\bar{X}_s-v_s\right) d W_s+\frac{\sigma^2}{2}\int_0^t \sum_{k=1}^d \operatorname{Diag}\left(\bar{X}_s-v_s\right)_{k k}^2 \partial_{k k}^2 \phi(\bar{X}_s)\; ds 
\end{align*}
and taking expectations and differentiating using Fubini's theorem, we have 
\begin{equation}
\begin{aligned}
\frac{d}{d t} \mathbb{E} \phi\left(\bar{X}_t\right)&=\mathbb{E}\left[ \nabla \phi\left(\bar{X}_t\right) \cdot\left(-\lambda \left(\bar{X}_t-v_t\right)+\beta u^*(\bar{X}_t)\right)\right] \\
&+\frac{\sigma^2}{2} \mathbb{E} \left[\sum_{k=1}^d \operatorname{Diag}\left(\bar{X}_t-v_t\right)_{k k}^2 \partial_{k k}^2 \phi(\bar{X}_t)\right].
\end{aligned}
\label{eq_e}
\end{equation}
Thus, the measure $\rho\in\mathcal{C}\left([0, T], \mathcal{P}_4\left(\mathbb{R}^d\right)\right)$  satisfies the Fokker-Planck equation in weak form
\begin{align*}
\frac{d}{d t} \int \phi(x) d \rho_t(x)&=-\lambda \int \left\langle x-v_t, \nabla \phi(x)\right\rangle d \rho_t(x)+ \beta\int \left\langle  u^*(x), \nabla \phi(x)\right\rangle d \rho_t(x)\nonumber\\ &+\frac{\sigma^2}{2} \int\sum_{k=1}^d \operatorname{Diag}\left(x-v_t\right)_{k k}^2 \partial_{k k}^2 \phi(x) \;d \rho_t(x).
\end{align*}
\textbf{Existence of strong solution} 
Start from a fixed function $v(t)\in\mathcal{C}\left([0, T], \mathbb{R}^d\right)$, we uniquely obtain a law $\rho_t$. Now, we define a mapping $\mathcal{T}: \mathcal{C}\left([0, T], \mathbb{R}^d\right) \rightarrow \mathcal{C}\left([0, T], \mathbb{R}^d\right)$ such that
$$ \mathcal{T} v=v_\alpha(\rho)\in \mathcal{C}\left([0, T], \mathbb{R}^d\right),
$$
where $v_\alpha(\rho_t)=\frac{1}{\int_{\mathbb{R}^d} \exp(-\alpha f(x)) \rho_t(x)} \int_{\mathbb{R}^d} x \exp\left(-\alpha f(x)\right) \rho_t(x)$.
To apply the Leray-Schauder fixed point theorem, we first need to prove the compactness and continuity of the mapping $t \mapsto v_\alpha\left(\rho_t\right)$ by showing the Hölder regularity in time $t$. 
Note that
$$\bar{X}_{t}-\bar{X}_s=\int_s^{t} -\lambda\left(\bar{X}_\tau-v_\tau\right)+\beta u^*(\bar{X}_\tau) \;d \tau+\int_s^{t} \sigma \operatorname{Diag}\left(\bar{X}_\tau-v_\tau\right) d W_\tau .$$
By \cref{assumption} and  Itô isometry, for any $0\leq s\leq t\leq T$, we have
\begin{align*}
 W^2_2\left(\rho_t, \rho_s\right)&=\mathbb{E}[|\bar{X}_t-\bar{X}_s|^2]\\
&\leq 3[\lambda^2(t-s)+\sigma^2]\mathbb{E} \int_s^t\left|\bar{X}_\tau-v_\tau\right|^2 d\tau+3(t-s)\beta^2 c_u\;\mathbb{E}\int_s^t(1+|\bar{X}_\tau|^2)d\tau\\
&\leq\left[3C(\lambda^2(t-s)+\sigma^2)+3(t-s)\beta^2 c_u\right]\mathbb{E}\int_s^t(1+|\bar{X}_\tau|^2)\;d\tau\\
&\leq \left[3C\left(\lambda^2 T+\sigma^2\right)+3T\beta^2c_u\right](1+R)|t-s|=:m^2|t-s|,
\end{align*}
The first inequality is due to the linear growth property of the drift term, $C$ is a constant.
Therefore, $ W_2\left(\rho_t, \rho_s\right) \leq m|t-s|^{\frac{1}{2}}$, for some constant $m>0$. Applying the result from \cref{lemma3}, we obtain
$$
\left|v_\alpha(\rho_t)-v_\alpha(\rho_s)\right| \leq m_0 W_2\left(\rho_t, \rho_s\right) \leq m_0 m|t-s|^{\frac{1}{2}},
$$
for some constant $m_0$ only depends on $\alpha, R$ and $L_f$. Since the mapping $t \mapsto v_\alpha\left(\rho_t\right)$
is Hölder continuous with exponent $1 / 2$,  therefore the compactness of $\mathcal{T}$ follows from the compact embedding
$\mathcal{C}^{0,1 / 2}\left([0, T], \mathbb{R}^d\right) \hookrightarrow\mathcal{C}\left([0, T], \mathbb{R}^d\right)$. Now we consider a set $\mathcal{V}=\{v \in \mathcal{C}\left([0, T], \mathbb{R}^d\right)\mid v=\tau \mathcal{T} v,\tau \in[0,1]\}$, the set $\mathcal{V}$ is non-empty due to the compactness of $\mathcal{T}$  (see \cite{gilbarg1977elliptic}). For any $v\in\mathcal{V}$, we have a corresponding process $\bar{X}$ satisfying \cref{sde_bar}, then there exists $\rho \in \mathcal{C}\left([0, T], \mathcal{P}_4\left(\mathbb{R}^d\right)\right)$ satisfying (\ref{FP}) such that $v=\tau v_\alpha(\rho)$.
By equation \cref{eq_e} and \cref{assumption} we have
\begin{align*}
\frac{d}{d t} \mathbb{E} \left|\bar{X}_t\right|^2&=\mathbb{E}\left[\sigma^2 \left|\bar{X}_t-v_t\right|^2 -2 \lambda \bar{X}_t\left(\bar{X}_t-v_t\right)+2\beta\bar{X}_t u^*(\bar{X}_t)\right]\\
&\leq \left( \sigma^2-2 \lambda+\beta(1+c_u)+ |\gamma|\right) \mathbb{E}\left|\bar{X}_t\right|^2 +\left(\sigma^2+|\gamma|\right)\left|v_t\right|^2+\beta c_u\\
&\leq (\sigma^2+\beta(1+c_u)+|\gamma|)\left(e^{\alpha(\bar{f}-\underline{f})}+1\right)\mathbb{E}\left|\bar{X}_t\right|^2+\beta c_u\\
&=:\bar{c} \,\mathbb{E}\left|\bar{X}_t\right|^2+\beta c_u
\end{align*}
with $\gamma=\lambda-\sigma^2$ and a constant $\bar{c}$. The last inequality is due to the fact $$|v_t|^2=\tau^2|\mathcal{T} v_t|^2=\tau^2\left|v_\alpha(\rho_t)\right|^2 \leq  e^{\alpha(\bar{f}-\underline{f})} \mathbb{E}\left|\bar{X}_t\right|^2.$$
From Gronwall's inequality we obtain
$
\mathbb{E}\left|\bar{X}_t\right|^2 \leq\left(\mathbb{E}\left|\bar{X}_0\right|^2+\beta c_u t\right) e^{\bar{c}\, t},
$ 
which implies the boundedness of the set $\mathcal{V}$. By applying the Leray-Schauder fixed point theorem, we conclude the existence of  fixed point for the mapping $\mathcal{T}$ and thereby the existence of  solution to \cref{sdee}.\\
\textbf{Uniqueness of strong solution} Suppose we have two fixed points $v$ and $\hat{v}$ of mapping $\mathcal{T}$ and their corresponding processes $X_t, \hat{X}_t$ satisfying \cref{sdee}, respectively.
Let $y_t:=X_t-\hat{X}_t$, then due to the Assumption \ref{assumption} we can easily obtain
\small
\begin{equation*}
y_t\leq y_0-\lambda \int_0^t y_s d s+\lambda \int_0^t(v_s-\hat{v}_s) d s
+\beta L_u\int_0^t|y_s|d s+\sigma\int_0^t(\operatorname{Diag}(y_s)+\operatorname{Diag}(\hat{v}_s-v_s))d W_s,
\end{equation*}
\normalsize
Squaring on both sides, taking the expectation and applying the Ito isometry yields
\begin{displaymath}
\mathbb{E}\left|y_t\right|^2 \leq  5\mathbb{E}\left|y_0\right|^2+5\left(\lambda^2 t+\sigma^2+
\beta^2L_u^2t\right) \int_0^t \mathbb{E}\left|y_s\right|^2 d s+5 (\lambda^2t+\sigma^2)\int_0^t\left|v_s-\hat{v}_s\right|^2 d s,
\end{displaymath}
Also we have
$$
|v_s-\hat{v}_s|=\left|v_\alpha(\rho_s)-v_\alpha(\hat{\rho}_s)\right| \leq m_0W_2\left(\rho_s, \hat{\rho}_s\right) \leq m_0 \sqrt{\mathbb{E}\left|y_s\right|^2},
$$
with $\rho=\text{law}(X)$ and
$\hat{\rho}=\text{law}(\hat{X})$. We further obtain
\begin{align*}
\mathbb{E}\left|y_t\right|^2 &\leq 5\mathbb{E}\left|y_0\right|^2+5\left(\left(1+m_0^2\right)(\lambda^2 t+\beta^2L_u^2t+\sigma^2)\right) \int_0^t \mathbb{E}\left|y_s\right|^2 d s\\
\frac{d}{dt}\mathbb{E}\left|y_t\right|^2 &\leq \beta(t) \mathbb{E}\left|y_t\right|^2,
\end{align*}
where $\beta(t):=5\left(\left(1+m_0^2\right)(\lambda^2 t+\beta^2L_u^2t+\sigma^2)\right)$, then by applying Gronwall's inequality we have 
$\mathbb{E}\left|y_t\right|^2\leq \mathbb{E}\left|y_0\right|^2\exp(\int_0^t\beta(s)ds)=0$ for all $t \in[0, T]$. Therefore, we obtain the uniqueness of the solution of \cref{sdee}.
\end{proof}
\section{Verifying \cref{assumption} for the quadratic case}
\label{appendix: qua}
In this section, we derive the control function $u^*$ in quadratic cases and show that it satisfies \cref{assumption}. If the objective function $f$ is quadratic, then we have
$$\min _{x \in \mathbb{R}^d} f(x)=\min _{x \in \mathbb{R}^d} x^\top Qx,\quad Q=Q^\top \succeq 0.$$ 
Consider the system in state-space form,
$\dot{y}(t)=u(t), \quad y(0)=x, \quad \mathcal{U}=\mathbb{R}^d$
and the cost function
$\mathcal{J}(u(\cdot), x)=\int_0^{\infty}e^{-\mu t}\left[y(t)^\top  Q y(t)+\frac{\epsilon}{2}u(t)^\top u(t)\right] d t.$
  Our goal is to find the optimal cost function $V(x)= \inf_{u\in\mathcal{U}} J(u(\cdot),x)$ which satisfies the HJB:
$$
-\mu V(x)+\min _u\left[x^\top  Qx+\frac{\epsilon}{2}u^\top u+DV(x)^\top u\right]=0.
$$
Since the optimal cost function is quadratic in this case. Let
$
V(x)=x^\top  S x,$ with $ S=S^\top  \succeq 0$, then the solution of the HJB equation is
$
\epsilon u+2 x^\top  S=0.
$
This yields the linear optimal control
\begin{equation}
u^*(x)=- \frac{2}{\epsilon}Sx.
\label{u_qua}
\end{equation}
Plugging $u^*$ back into the HJB, we have
$
0=x^\top [-\mu S+Q-\frac{2}{\epsilon}S^\top  S] x,
$
which implies $\mu S+\frac{2}{\epsilon}S^\top  S=Q$. The solution to the system $\dot{y}=u^*(y(t))=- \frac{2}{\epsilon}Sy(t)$ is then $y(t)=e^{-\frac{2}{\epsilon}St}y(0)$. Finally, it is easy to verify that 
the optimal control \cref{u_qua} fulfills the \cref{assumption}(3) with $L_u=\frac{2}{\epsilon}\|S\|_\infty$ and $c_u=\frac{4}{\epsilon^2}\lambda_{\text{max}}(S^\top S).$

\bibliographystyle{siamplain}
\bibliography{references}

\end{document}